\documentclass[a4paper,11pt]{amsart}
\usepackage{amssymb,mathrsfs}
\usepackage{amsthm}
\usepackage{amsmath}
\usepackage{latexsym}
\usepackage{fancyhdr}
\usepackage{ascmac}
\usepackage{color}
\usepackage[all]{xy}
\usepackage{amscd}

\theoremstyle{plain}
\newtheorem{thm}{Theorem}[section]
\newtheorem{prop}[thm]{Proposition}
\newtheorem{cor}[thm]{Corollary}
\newtheorem{lem}[thm]{Lemma}
\newtheorem{dfn}[thm]{Definition}

\newtheorem{rmk}[thm]{Remark}
\newtheorem{qn}[thm]{Question}

\makeatletter

\@addtoreset{equation}{section}
\makeatother

\newcommand{\bQ}{\overline{\mathbb{Q}}}

\newcommand{\bZ}{\overline{\mathbb{Z}}}

\newcommand{\bF}{\overline{\mathbb{F}}}

\newcommand{\C}{\mathbb{C}}
\newcommand{\R}{\mathbb{R}}
\newcommand{\Q}{\mathbb{Q}}

\newcommand{\Z}{\mathbb{Z}}

\newcommand{\F}{\mathbb{F}}
\newcommand{\GSp}{{\rm GSp}}

\newcommand{\lra}{\longrightarrow}

\newcommand{\sgn}{{\rm sgn}}
\newcommand{\A}{\mathbb{A}}

\renewcommand{\O}{\mathcal{O}}

\newcommand{\br}{\overline{\rho}}

\newcommand{\diag}{{\rm diag}}
\newcommand{\ds}{\displaystyle}

\newcommand{\uk}{\underline{k}}
\newcommand{\uw}{\underline{w}}

\newcommand{\gl}{{\rm GL}}

\newcommand{\e}{\varepsilon}

\newcommand{\bs}{\backslash}

\newcommand{\ot}{\overline{\tau}}

\title[Automorphy of mod 2 Galois representations associated to the quintic Dwork family and 
reciprocity of some quintic trinomials] 
{Automorphy of mod 2 Galois representations associated to the quintic Dwork family and 
reciprocity of some quintic trinomials }
\author{Nobuo Tsuzuki${}^{\dagger}$ and Takuya Yamauchi${}^{\dagger\dagger}$}

\thanks{The author${}^\dagger$
is partially supported by JSPS KAKENHI Grant Number (A) No.18H03667, and the author${}^{\dagger\dagger}$
by JSPS KAKENHI Grant Number (B) No.19H01778.}

\address{Nobuo Tsuzuki \\ 
Mathematical Inst. Tohoku Univ.\\
 6-3,Aoba, Aramaki, Aoba-Ku, Sendai 980-8578, JAPAN}
\email{tsuzuki@math.tohoku.ac.jp}

\address{Takuya Yamauchi \\ 
Mathematical Inst. Tohoku Univ.\\
 6-3,Aoba, Aramaki, Aoba-Ku, Sendai 980-8578, JAPAN}
\email{takuya.yamauchi.c3@tohoku.ac.jp}

\begin{document}

\maketitle

\begin{abstract}
In this paper, we determine mod $2$ Galois representations 
$\br_{\psi,2}:G_K:={\rm Gal}(\overline{K}/K)\lra {\rm GSp}_4(\F_2)$ 
associated to the mirror motives of rank 4 with pure weight 3 coming from the 
Dwork quintic family 
$$X^5_0+X^5_1+X^5_2+X^5_3+X^5_4-5\psi X_0X_1X_2X_3X_4=0,\ \psi\in K$$
 defined over a number field $K$ under the irreducibility condition of the quintic trinomial $f_\psi$ below. 
Applying this result, when $K=F$ is a totally real field, for some at most qaudratic totally real extension $M/F$,   
we prove that $\br_{\psi,2}|_{G_M}$ is associated to a Hilbert-Siegel modular Hecke eigen cusp form  
for ${\rm GSp}_4(\A_M)$ of parallel weight three.  

In the course of the proof, we observe that the image of such a mod $2$ representation is governed by reciprocity of 
the quintic trinomial $$f_\psi(x)=4x^5-5\psi x^4+1,\ \psi\in K$$ whose decomposition field is generically of type 
5-th symmetric group $S_5$. 
This enable us to use results on the modularity of 
2-dimensional, totally odd Artin representations of ${\rm Gal}(\overline{F}/F)$ due to Shu Sasaki 
and several Langlands functorial lifts for Hilbert cusp forms. 
Then, it guarantees the existence of a desired Hilbert-Siegel modular cusp form of parallel weight 
three matching with the Hodge type of the compatible system in question. 
A twisted version is also discussed and it is related to general quintic trinomials. 
\end{abstract}

\tableofcontents

\section{Introduction}
Let $K$ be a number field in an algebraic closure $\bQ$ of $\Q$ and $p$ be a prime number.  
Fix an embedding $\bQ\hookrightarrow \C$ and an isomorphism $\bQ_p\simeq \C$. 
Let $\iota=\iota_p:\bQ\lra\bQ_p$ be am embedding which is compatible with $\bQ\hookrightarrow \C$ and 
$\bQ_p\simeq \C$ fixed right before. 

Let $\br:G_K:={\rm Gal}(\bQ/K)\lra {\rm GL}_n(\bF_p)$ be an irreducible, continuous representation which is 
said to be a mod $p$ Galois representation in this paper.  
Number theorists have been expected that $\br$ is automorphic, that is 
there exist an algebraic automorphic representation $\pi$  of $\gl_n(\A_K)$ (see \cite{BG}) and 
its conjectural $p$-adic Galois representation $\rho_{\pi,\iota}:G_K\lra \gl_n(\bQ_p)$ such that 
$\br\simeq \br_{\pi,\iota}$ where $\br_{\pi,\iota}$ is the reduction of $\rho_{\pi,\iota}$ modulo 
the maximal ideal of $\bZ_p$ after choosing a suitable $G_K$-invariant $\bZ_p$-lattice. The automorphy of $\br$ is understood as the Serre conjecture for $(G_K, \mathrm{GL}_n)$ and 
it has been broadly studied over the past few decades (see \cite{GHS} and the references there) in more general setting. 
To prove automorphy of a given geometric $p$-adic Galois representation, it would be necessary for its 
reduction to be automorphic in applying current techniques. 
In general, no concrete way has been known to find such a $\pi$ for $\br$ except for 
some cases. Easier one among known cases requires that $p$ and 
the image of $\br$ are small so that one can apply automorphy of Artin representations. 
Therefore, it would be important to find geometric objects whose residual Galois representations 
have small images in which case we can apply the known cases of automorphic Artin representations.  

In this paper, we study mod $2$ Galois representations associated to the mirror family associated to 
the Dwork quintic family and 
its automorphy in connection with cuspidal automorphic representations of $\mathrm{GSp}_4$. 

Let us fix the notation to explain our results. 
For each $\psi\in K$ with $\psi^5\neq 1$, let us consider the Calabi-Yau threefold defined by 
\begin{equation}\label{x}
X_\psi: X^5_0+X^5_1+X^5_2+X^5_3+X^5_4-5\psi X_0X_1X_2X_3X_4=0
\end{equation}
as a smooth projective hypersurface in $\mathbb{P}^4$ with the coordinates $[X_0:X_1:X_2:X_3:X_4]$.  
It is called the Dwork quintic family when we view it as a family with one parameter $\psi$. 
Let $Y_\psi$ be the singular mirror symmetry of $X_\psi$ which is defined by 
the closure of the smooth affine variety 
\begin{equation}\label{y}
  U_\psi:    x_1 + x_2 + x_3 + x_4 + \frac{1}{x_1x_2x_3x_4} - 5\psi = 0. 
\end{equation}
in the projective toric variety $\mathbb P_\Delta$ which will be introduced in Section \ref{MV}.  
Let us take a smooth mirror symmetry $W_\psi$ of $X_\psi$ introduced in \cite{COR} such that 
$W_\psi$ is a crepant resolution of $Y_\psi$ (see \cite{Ba}). 
Then $W_\psi$ has good reduction at each finite place $v$ of $K$ such that $v\nmid 5$ and $\psi^5-1$ is 
a $v$-adic unit (see Section \ref{MV}). 
 Such a mirror symmetry is not unique, but 
it is unique as a pure motif. The Hodge-diamond of $W_\psi$ is given by 
$$
\begin{array}{cccccccc}
H^0&&&& 1 &&& \\
H^1&&&0&  &  0&& \\
H^2&&0&& 101 &  &0& \\
H^3&1&&1&  & 1 &&1  \\
H^4&&0&& 101 &  &0& \\
H^5&&&0&  &  0&& \\
H^6&&&& 1 &&& \\
\end{array}
$$
where $H^\ast$ stands for the (complex) de Rham cohomology. 
Put $$V_{\psi,2}:=H^3_{\mathrm{et}} (W_{\psi, \overline{\mathbb Q}}, \Q_2)$$ 
where $H^\ast_{\mathrm{et}}$ stands for etale cohomology, and 
let $\langle \ast,\ast \rangle:V_{\psi,2}\times V_{\psi,2}\lra \Q_2(-3)$ be the $G_K$-equivalence, perfect, 
alternating pairing defined by the Poincar\'e duality. 
It yields a $2$-adic Galois representation 
$$\rho_{\psi,2}=\rho_{\psi,\iota_2}:G_K\lra {\rm GSp}(V_{\psi,2},\langle \ast,\ast \rangle)\simeq {\rm GSp}_4(\Q_2)$$
where the algebraic group $\mathrm{GSp}_4= \mathrm{GSp}_J$ is the symplectic similitude group in $\mathrm{GL}_4$ 
associated to $J=\begin{pmatrix} 0_2& s\\-s &0_2\end{pmatrix},\ 
s=\begin{pmatrix} 0& 1\\1 &0\end{pmatrix}$. 
Let us choose a $G_K$-stable lattice $T_{\psi,2}$ over $\Z_2$ of $V_{\psi,2}$ so that the above alternating pairing 
preserves the integral structure with respect to $T_{\psi,2}$. Put $\overline{T}_{\psi,2}=
T_{\psi,2}\otimes_{\Z_2}\F_2$. Thus, it yields a mod 2 
Galois representation 
\begin{equation}\label{mod2}
  \br_{\psi,2}:G_K\lra 
{\rm GSp}(\overline{T}_{\psi,2},\langle \ast,\ast \rangle_{\F_2})\simeq {\rm GSp}_4(\F_2)
\end{equation} 
depending on the choice of $T_{\psi,2}$.   
We view it a representation to ${\rm GL}_4(\F_2)$ via the natural inclusion ${\rm GSp}_4(\F_2)\subset 
{\rm GL}_4(\F_2)$. Therefore, one can consider the seimisimplification $\br^{{\rm ss}}_{\psi,2}$ of $\br_{\psi,2}$. 
A priori, it has the image in ${\rm GL}_4(\F_2)$, but we can suitably choose a symplectic basis so that 
it takes the values in ${\rm GSp}_4(\F_2)$. Hence we have a semisimple mod 2 Galois representation 
$$\br^{{\rm ss}}_{\psi,2}:G_K\lra  {\rm GSp}_4(\F_2)$$
associated to $W_\psi$. As is explained in Remark \ref{c6}, that 
$H^3_{\mathrm{et}} (W_{\psi, \overline{\mathbb Q}}, \Z_2)$ is torsion free and one can choose it as $T_{\psi,2}$. 
As explained in Section \ref{mod2}, let us fix an isomorphism ${\rm GSp}_4(\F_2)\simeq S_6$.  
  
Let us introduce the following quintic trinomial
\begin{equation}\label{tri}
f_\psi(x):=4x^5-5\psi x^4+1 \in K[x]
\end{equation} 
and denote by $K_{f_\psi}$ the decomposition field of $K$ in $\bQ$. 
Notice that the discriminant of $f_\psi$ is given by $2^85^5(1-\psi^5)$. 
First we prove the following result:
\begin{thm}\label{image}$($Theorem \ref{s5}, Proposition \ref{abs-irr}, and Theorem \ref{completed-version}$)$ Assume that $f_\psi$ is irreducible over $K$. Then it holds that 
\begin{enumerate}
\item 
${\rm Im}(\br_{\psi,2})\simeq {\rm Gal}(K_{f_\psi}/K)$ and ${\rm Im}(\br_{\psi,2})$ contains no element of type $(3,3)$  
under the fixed isomorphism ${\rm GSp}_4(\F_2)\simeq S_6$. 
In particular, the image is regarded as a subgroup of the 5-th symmetric group $S_5$ whose order is divisible by five; 
\item $\br_{\psi,2}$ is irreducible over $\F_2$ and hence $\br_{\psi,2}\simeq \br^{{\rm ss}}_{\psi,2}$. Further, it is absolutely irreducible in which case the image is isomorphic to 
$F_{20}=C_4\ltimes C_5,A_5$ or $S_5$ unless 
the image is isomorphic to $C_5$ or $D_{10}=C_2\ltimes C_5 ($see Section \ref{subS6} for the subgroups$)$;
\item when $K=\Q$, ${\rm Im}(\br_{\psi,2})\simeq S_5$ unless $\psi=0$ in which case ${\rm Im}(\br_{0,2})\simeq F_{20}$. 
\end{enumerate}
\end{thm}
This theorem explains reciprocity of the quintic trinomial $f_\psi$ is governed by 
the reduction of the Frobenius polynomial $P_{\psi, v}$ defined as below, as a polynomial in $\F_2[t]$. 
For each good finite place $v$ of $K$ with the residue field $\F_v$ of $q_v$ elements, 
put
$$
P_{\psi, v}(t) := \det(1_4-t\rho_{\psi, 2}({\rm Frob}_v)) = 1 - a_{\psi, v}t + b_{\psi, v}t^2 - q_v^3a_{\psi, v}t^3 + q_v^6t^4 \in \Z_2[t]
$$
which turns to be a Weil polynomial over $\Z$. If $n(f_\psi, q_v)$ denotes the cardinal of solutions of the equation $f_\psi(x)=0$ in $\F_v$, 
then the reciprocity in $f_{\psi}$ is the congruence 
\begin{equation}\label{Rec}
           a_{\psi, v}\, \equiv\, n(f_\psi, q_v) + 1\, (\mathrm{mod}\, 2)
\end{equation}
for all good finite places $v$ not lying over $2$ (see Proposition \ref{c0}).  
There are only four cases for $P_{\psi,v}(t) \ ({\rm mod}\ 2)$ and explicitly they are given as follows:
$$1+t^4,\ 1+t + t^3+t^4,\ 1+t+t^2+t^3+t^4,\ 1+t^2+t^4.$$ 
Under a fixed isomorphism ${\rm GSp}_4(\F_2)\simeq S_6$ explained in Section \ref{GSp4vsS6}, the reciprocity (\ref{Rec}) and Chebotarev's density theorem 
reveal that $K_{f_\psi}$ and $L_\psi:=\bQ^{{\rm Ker}(\br_{\psi,2})}$ are intervened each other and it yields that  
$\br_{\psi,2}$ takes values in a subgroup isomorphic into $S_5$. 
A key to analyze the image is to observe 
the disappearance of $(123)(456) \in S_6$ in the image of $\overline{\rho}_{\psi, 2}$, 
whose characteristic 
polynomial corresponds to $1+t^2+t^4\, \equiv\, (1+t+t^2)^2\, (\mathrm{mod}\, 2)$. 
Our result would be, in some sense, new in computing the image of residual Galois representations coming from 
higher dimensional varieties which would not be kwnon a priori to be automorphic 
though there are singinificant works as in \cite{E}. 

Applying the above theorem, one can describe $\br_{\psi,2}$ in terms of 2-dimensional mod 2 Galois 
representations. Artin representations with $S_5$-image have been studied in \cite{Ca} and this result is 
suggestive to study our representations in a similar way though we need more careful analysis. In fact, 
${\rm GSp}_4(\F_2)={\rm Sp}_4(\F_2)\simeq {\rm SO}(4)(\F_2)$ implies there are two cases of mod 2 automorphy such that  
one is for orthogonal automorphic representations and the other one is for symplectic automorphic representations. 
Notice that before taking the reduction modulo 2, the $2$-adic Galois representation $\rho_{\psi,2}$ is symplectic. 
Therefore, we better find symplectic automorphic representations for $\br_{\psi,2}$. However, there is an issue that 
it is revealed later that when ${\rm Im}(\br_{\psi,2})\simeq S_5$ (generic case), $\br_{\psi,2}$ can be realized 
as a twisted tensor product and it is naturally connected to orthogonal automorphic representations. 
Thus, it seems to be hard to find any symplectic automorphic lift from this.  
To overcome this issue we restrict $\br_{\psi,2}$ to $G_M$ where $M/K$ is the quadratic extension corresponding to 
the kernel of $G_K\lra {\rm Im}(\br_{\psi,2})\simeq S_5\stackrel{\sgn}{\lra}\{\pm 1\}$. Then it turns out that 
$\br_{\psi,2}|_{G_M}$ has the symmetric cubic structure which can be easily liftable to a symplectic representation.  
  
We can combine the above main theorem with many ingredients obtained in \cite{Sasaki}, the results of algebraic quaternionic forms, and 
known Langlands functorial lifts  
to prove the automorphy of our mod 2 Galois representations when $K=F$ is a totally real field. 
Let $\mathcal{H}_2$ be the Siegel-upper half space of degree 2. 
\begin{thm}\label{mt-automorphy}$($Theorem \ref{mt-a5} and Theorem \ref{mt-F20}$)$
Suppose that $K=F$ is a totally real field and $f_\psi$ is irreducible over $F$. 
Assume that ${\rm Gal}(F_{f_\psi}/F)\simeq F_{20},\ A_5$ or $S_5$ and each complex conjugation in ${\rm Gal}(F_{f_\psi}/F)$ 
corresponds to an element of type $(2,2)$ under this isomorphism, equivalently, $\sigma(\psi)<1$ for any embedding $\sigma:F\hookrightarrow \R$. Further assume $[F:\Q]$ is even if ${\rm Gal}(F_{f_\psi}/F)\simeq A_5$. Let $M/F$ be the totally real quadratic extension 
associated to the kernel of $G_F\lra {\rm Im}(\br_{\psi,2})\lra S_5\stackrel{\sgn}{\lra}\{\pm 1\}$ if ${\rm Gal}(F_{f_\psi}/F)\simeq S_5$ and 
$M=F$ otherwise. Put $d=[M:\Q]$.  
Then there exists a holomorphic Hilbert-Siegel Hecke eigen cusp form $h$ on $\mathcal{H}^d_2$ of parallel weight three such that   
$$\br_{\psi,2}|_{G_M}\simeq \br_{h,2}$$
where $\br_{h,2}$ is the reduction of $2$-adic representation $\rho_{h,2}$ associated to $h$ {\rm (}see \cite{Mok} for 
the construction of $\rho_{h,2})$.  
\end{thm}
To consider the restriction of $\br_{\psi,2}$ to $G_M$ would be harmless in proving the automorphy of $\rho_{\psi,2}$ since 
the automorphy of $\rho_{\psi,2}|_{G_M}$ implies the one of $\rho_{\psi,2}$ by a well-known descend argument. 

The assumption on $[F:\Q]$ when ${\rm Gal}(F_{f_\psi}/F)\simeq A_5$ is necessary to fully 
apply Jacquet-Langlands correspondence. If we succeed to make use of a level raising argument, 
we would be able to remove this assumption.  

As observed, $\br_{\psi,2}$ can be viewed as an orthogonal representation via 
${\rm GSp}_4(\F_2)={\rm Sp}_4(\F_2)={\rm SO}(4)(\F_2)$ where the orthogonal group ${\rm SO}(4)$ is defined by using 
$\begin{pmatrix} 0_2& s\\s &0_2\end{pmatrix},\ 
s=\begin{pmatrix} 0& 1\\1 &0\end{pmatrix}$. Therefore, it would be also interesting to discuss 
orthogonal automorphy though it no longer has any relation to Dwork family but may be related to arithmetic around 
the arithmetic variety associated to ${\rm GL}_4$ (cf. \cite{AGMC}).  
\begin{thm}\label{o-automorphy}$($Theorem \ref{mt-s5}$)$
Let $\br:G_F\lra {\rm SO}(4)(\F_2)$ be an irreducible mod 2 Galois representation.  
Assume that ${\rm Im}(\br)\simeq S_5$ and each complex conjugation in $G_F$ 
corresponds to an element of type $(2,2)$ under this isomorphism. 
Then there exists a cuspidal automorphic representation $\pi$ of ${\rm GO(2,2)}(\A_F)$ which can be 
transfered to a regular algebraic essentially self-dual cuspidal automorphic representation $\Pi$ of ${\rm GL}_4(\A_F)$ of weight zero such that   
$$\br\simeq \br_{\Pi,2}$$
where $\br_{\Pi,2}$ is the reduction of $2$-adic Galois representation $\rho_{\Pi,2}$ associated to $\Pi$ {\rm (}see Section 1 of \cite{BGHT} for 
the construction of $\rho_{\Pi,2})$.  
\end{thm}
In the course of the proof of the above theorem, we will use some arithmetic properties of 
non-paritious Hilbert modular forms to construct cuspidal automorphic representations of ${\rm GL}_4(\A_F)$ of weight zero. 
This would be a new application of interesting results in \cite{DLP} and \cite{P}.

\begin{rmk}\label{LEVEL}In above main theorems, with our best knowledge for 
the moment, we can not specify the levels of Siegel-Hilbert forms due to 
the lack of level lowering results $($see Remark \ref{level}$)$.
\end{rmk}

\begin{rmk}\label{PA}The compatible system associated to our Dwork family $X_\psi$ over a CM field 
is known to be potentially automorphic $($\cite{BL},\cite{PT}$)$. Hence we can study possible levels of 
the lifts potentially at 
a suitable $\ell$-adic component and then we may convert some arithmetic information in the levels to those of the $2$-adic component and its reduction modulo 2.   
\end{rmk}

\begin{rmk}\label{GeeT}
In the course of the proof of Theorem \ref{mt-automorphy}, to obtain holomorphic forms we will apply 
the results in \cite{GeeT}. 
Note that the results in \cite{GeeT} are conditional on the trace formula 
$($see the second paragraph in p.472 of \cite{GeeT}$)$. 
However, if we do not require holomorphic forms for automorphy, the proof of \ref{mt-automorphy} shows unconditionally that there exists a 
global generic regular algebraic cuspidal automorphic representation $\Pi$ of ${\rm GSp}_4(\A_M)$ such that 
$\br_{\Pi,2}\simeq \br_{\psi,2}|_{G_M}$.   
\end{rmk}

This paper will be organized as follows. 
In Section 2, we study basic facts for mod 2 Galois representations to ${\rm GSp}_4(\F_2)$. 
We devote Section 3 through Section 4 to study automorphic forms in question and various congruences between 
several automorphic forms. 
We devote Section 5 through Section 6 to determine the image of $\br_{\psi,2}$ for which the results are 
important and play important roles to automorphy in the previous sections.  
In the last section, we give some observations in the possibility of the image of $\br_{\psi,2}$ related to 
some rational points on Diophantine equations in which we use some technique being 
reminiscent of Coleman-Chabauty method though our method is more elementary and is of independent interest. 
Finally, we discuss a variant of our main theorems by introducing 
a twist of Dwork quintic family.  

Let us give a guide for readers. 
The determination of the mod 2 Galois representations are more friendlier than other contents to readers and 
they may read Section 3 and Section 6 assuming the results in Section 5. 
Section 4 and Section 7 are readable independently apart from the other sections. 
The main results in Section 1 is slightly generalized to a twisted version in Section 8.

In \cite{TY}, we will study automorphy of 2-adic Galois representations of rank four with 
the residual image $A_5$ by using idea of \cite{T1} and \cite{KT}.

\textbf{Acknowledgment.} This work started when the second author discussed with the first author about the quintic Dwork family. 
We thank the Tohoku university for making this opportunity and incredible hospitality. 
The second author would like to thank Gabor Wiese for helpful comments. 
 
\section{$\mathrm{GSp}_4$}\label{gsp4}
Let us fix some notation for the smooth group scheme 
$\mathrm{GSp}_4=\mathrm{GSp}_J$ over $\Z$ which is defined 
as the symplectic similitude group in $\mathrm{GL}_4$ associated to $J=\begin{pmatrix} 0_2& s\\-s &0_2\end{pmatrix},\ 
s=\begin{pmatrix} 0& 1\\1 &0\end{pmatrix}$. Explicitly, 
$$\mathrm{GSp}_4=\{X\in \mathrm{GL}_4\ |\ {}^tXJX=\nu(X)J,\ {}^\exists \nu(X)\in \mathrm{GL}_1\}.$$ 
Put $\mathrm{Sp}_4={\rm Ker}(\nu:\mathrm{GSp}_4\lra \mathrm{GL}_1,X\mapsto \nu(X))$. 
Let $B$ be the upper Borel subgroup in $\mathrm{GSp}_4$ with the Levi decomposition $B=TN$. 
Let $P$ be the Siegel parabolic subgroup containing  $B$ and denote by $P=M_PN_P$ its Levi decomposition. 
Let $Q$ be the Klingen parabolic subgroup containing  $B$ with the Levi decomposition $Q=M_QN_Q$. 
All non-trivial (rational) parabolic subgroups containing $B$ are either of $B,P$, or $Q$. 
An explicit form in each case can be found in p.28-29 of \cite{RS}. 
Finally, we define the endoscopic subgroup $H$ of $\mathrm{GSp}_4$ consisting of all elements 
$\begin{pmatrix}
x & 0 & 0 & y \\
0 & a & b & 0 \\
0 & c & d & 0 \\
z & 0 & 0 & w 
\end{pmatrix}$ with $xw-yz=ad-bc$. 
\begin{lem}\label{str}Let $\F_q$ be a finite field of characteristic $2$. 
Then it holds that 
\begin{enumerate}
\item $B(\F_q)\simeq (\F^\times_q)^3\ltimes (\F_q\ltimes \F^3_q)$;
\item $P(\F_q)\simeq (\F^\times_q\times{\rm GL}_2(\F_q))\ltimes \F^3_q$;
\item $Q(\F_q)\simeq (\F^\times_q\times{\rm GL}_2(\F_q))\ltimes \F^3_q$. 
\end{enumerate}
\end{lem}
\begin{proof}Since the characteristic of $\F_q$ is 2, it is easy to see that $N_Q(\F_q)$ is abelian. 
The third claim easily follows from the structure of $Q$. 
For $B$, $N(\F_q)$ consists of all elements of form $n(x,\lambda,\mu,\kappa):=
\begin{pmatrix}
1 & 0 & 0 & 0 \\
0 & 1 & x & 0 \\
0 & 0 & 1 & 0 \\
0 & 0 & 0 & 1 
\end{pmatrix}\begin{pmatrix}
1 & \lambda & \mu & \kappa \\
0 & 1 & 0 & \mu \\
0 & 0 & 1 & -\lambda \\
0 & 0 & 0 & 1 
\end{pmatrix}$. Let $N_1=\{n(0,\lambda,\mu,\kappa)\in N(\F_q) \}$. It is easy to check that 
it is a normal abelian subgroup of $N(\F_q)$ and it follows from this that $N_1\simeq \F^3_q$. 
Clearly, $N(\F_q)/N_1\simeq \{n(x,0,0,0,)\in N(\F_q) \}\simeq \F_q$. Hence we have the claim for $B$. 
The remaining case $P$ is similar to the case of $Q$.  
\end{proof}
Notice that the statement of Lemma \ref{str} is false for $B$ and $Q$ when the characteristic of 
the base field is different from 2. 
\section{mod 2 Galois representations to ${\rm GSp}_4(\F_2)$}\label{mod2}
In this section, we study some elementary properties of mod 2 Galois representations to 
${\rm GSp}_4(\F_2)$. We denote by $S_n$ $n$-th symmetric group.  

\subsection{An identification between ${\rm GSp}_4(\F_2)$ and $S_6$}\label{GSp4vsS6}
Let $s:\F^6_2\lra \F_2$ be the linear functional defined by $s(x_1,\ldots,x_6)=x_1+\cdots+x_6$ and 
put $V=\{x\in \F^6_2\ |\ s(x)=0\}$ and $W=V/U$ where $U=\langle (1,1,1,1,1,1) \rangle$. 
Let us consider the bilinear form on $\F^6_2$ given by the formula 
$$\langle x,y \rangle=x_1y_1+\cdots +x_6y_6,\ x,y\in \F^6_2.$$
It induces a non-degenerate, alternating pairing $\langle \ast,\ast \rangle_W$ on $W$ where 
being alternating means $\langle x,x \rangle_W=0$ for each $x\in W$. 
The symmetric group $S_6$ naturally act on $\F^6_2$ and it yields a group homomorphism 
$$\varphi:S_6\lra {\rm GSp}(W,\langle \ast,\ast \rangle_W).$$
It is easy to see that the action of $S_6$ on $W$ is faithful and 
$\varphi$ turns out to be isomorphism by counting all elements of 
both sides. 
There is no canonical choice for basis of $W$ but we have a canonical choice with respect to the 
endoscopic subgroup $H$. 
Let us act the symmetric group $S_3$ on $Y=\{(x_1,x_2,x_3)\in \F^3_2\ |\ x_1+x_2+x_3=0\}$. 
Similarly, we have an isomorphism $$\psi:S_3\lra {\rm GL}(Y)={\rm GL}(Y,\langle \ast,\ast\rangle_Y).$$
The isomorphism $$\eta:Y\oplus Y\lra W,\ ((x_1,x_2,x_3),(y_1,y_2,y_3))\mapsto 
(x_1,x_2,x_3,y_1,y_2,y_3)\ {\rm mod}\ U$$ 
as a $\F_2$-vector space satisfies the relation 
$$\langle u_1,v_1\rangle_Y+\langle u_2,v_2\rangle_Y=
\langle \eta(u_1,v_1),\eta(u_2,v_2)\rangle_W$$  
and further we have the following commutative diagram: 
$$
\xymatrix{
S_3\times S_3 \ar[d]_{\iota} \ar[r]^{(\psi,\psi)\hspace{7mm}} & {\rm GL}(Y)\times {\rm GL}(Y) \ar[d]_\eta \\
 S_6 \ar[r]^{\varphi\hspace{7mm}}  &  {\rm GSp}(W,\langle \ast,\ast \rangle_W)
} 
$$
where the left vertical arrow $\iota$ is defined by the identification 
$$S_3\times S_3={\rm Aut}(\{1,2,3\})\times 
{\rm Aut}(\{4,5,6\})\subset {\rm Aut}(\{1,2,3,4,5,6\})=S_6.$$ 
For instance, choose  $f_1=(1,1,0),f_2=(1,0,1)$ as a basis of $Y$. Then we see that 
\begin{equation}\label{basis}
e_1=\eta(f_1,0),\ e_2=\eta(0,f_1),\ e_3=\eta(0,f_2),\ e_4=\eta(f_2,0)
\end{equation}
make up a basis of $W$. 
Using these basis,  
we have the following commutative diagram: 
$$
\xymatrix{
S_3\times S_3 \ar[d]_{\iota} \ar[r] & {\rm GL}_2(\F_2)\times {\rm GL}_2(\F_2) \ar[d] \\
 S_6 \ar[r]  &  {\rm GSp}_4(\F_2)
} 
$$
such that the image of the right vertical arrow coincides with $H(\F_2)$. 
For instance, the element $(123)(345)$ of order 3 is represented by 
$$
\left(\begin{array}{cccc}
 1 &0 &0 &1 \\
 0 &1 &1 &0 \\
 0 &1 &0 &0 \\
 1 &0 &0 &0
\end{array}
\right). 
$$
To end this subsection, we list up all representation matrices with respect to the basis (\ref{basis}) for each conjugacy class of $S_5$. Here $S_5$ is regarded as $\{\sigma\in S_6\ |\ \sigma(6)=6\}$. 
They are 
$$\tau_1=e,\ \tau_2=(12),\ \tau_{22}=(12)(34),\ \tau_3=(123),\ \tau_{32}=(123)(45),\ \tau_4=(1234),\ \tau_5=(12345)$$
where $e$ stands for the identity element in $S_5$. 
In this order, we have seven matrices respectively, 
$$A_{\tau_1}=I_4,\ 
A_{\tau_2}=\left(\begin{array}{cccc}
 1 &0 &0 &1 \\
 0 &1 &0 &0 \\
 0 &0 &1 &0 \\
 0 &0 &0 &1
\end{array}
\right),\ 
A_{\tau_{22}}=\left(\begin{array}{cccc}
 1 &1 &1 &0 \\
 0 &0 &1 &1 \\
 0 &1 &0 &1 \\
 0 &0 &0 &1
\end{array}
\right),\ 
A_{\tau_{3}}=\left(\begin{array}{cccc}
 0 &0 &0 &1 \\
 0 &1 &0 &0 \\
 0 &0 &1 &0 \\
 1 &0 &0 &1
\end{array}
\right),$$
$$A_{\tau_{32}}=\left(\begin{array}{cccc}
 0 &0 &0 &1 \\
 0 &1 &1 &0 \\
 0 &0 &1 &0 \\
 1 &0 &0 &1
\end{array}
\right),\ 
A_{\tau_{4}}=\left(\begin{array}{cccc}
 1 &1 &1 &0 \\
 1 &0 &1 &1 \\
 1 &1 &0 &1 \\
 1 &0 &0 &1
\end{array}
\right),
A_{\tau_{5}}=\left(\begin{array}{cccc}
 1 &1 &1 &0 \\
 0 &1 &1 &0 \\
 1 &1 &0 &1 \\
 1 &0 &0 &1
\end{array}
\right).$$
The eigen-polynomials are given by 
$$f_{\tau_1}(t)=f_{\tau_2}(t)=f_{\tau_{22}}(t)=f_{\tau_4}(t)=(1+t)^4,\ f_{\tau_{3}}(t)=f_{\tau_{32}}(t)=(1+t)^2(1+t+t^2),$$
and $f_{\tau_{5}}(t)=1+t+t^2+t^3+t^4$. 

\subsection{Certain subgroups of $S_6$}\label{subS6}
In our purpose, we are concerned with any subgroup of $S_6$ whose order is divisible by 5.  
Applying GAP \cite{GAP}, up to conjugacy, we have  
$$C_5,D_{10}:=C_2\ltimes C_5,\ F_{20}:=C_4\ltimes C_5,\ A_5,\ S_5,\ A_6,\ S_6.$$
Clearly, $C_5\subset D_{10}\subset F_{20}$ and $D_{10}\subset A_5$ but $F_{20}\not\subset A_5$. 
Here $C_n$ stands for the cyclic group of order $n$ and 
$A_n$ is the alternating group of degree $n$. 
For instance, in terms of the basis (\ref{basis}), 
the generators $\sigma=(23)(56),\ \tau=(25463)$ of $D_{10}$ in $S_6$ have the following matrix representations  
$$
J=\begin{pmatrix}
0 & 0 & 0 & 1 \\
0 & 0 & 1 & 0 \\
0 & 1 & 0 & 0 \\
1 & 0 & 0 & 0 
\end{pmatrix},\ 
\begin{pmatrix}
0 & 0 & 0 & 1 \\
0 & 0 & 1 & 1 \\
0 & 1 & 0 & 0 \\
1 & 1 & 0 & 1 
\end{pmatrix}
\in {\rm GSp}_4(\F_2)$$
respectively.  

On the other hand, $F_{20}$ can be realized as 
$$F_{20}=\langle \sigma=(12345),\ \tau=(1243) \rangle\subset S_5.$$
Notice that $\tau^2=(14)(23)$ is of type $(2,2)$. 

\subsection{An identification between $A_5$ and ${\rm SL}_2(\F_4)$}\label{A5}
It is well-known that $A_5\simeq {\rm SL}_2(\F_4)$ as a group. 
Notice that ${\rm PGL}_2(\F_4)\simeq {\rm PSL}_2(\F_4)={\rm SL}_2(\F_4)$ since $(\F^\times_4)^2=\F^\times_4$. 
Then we have an embedding 
$${\rm SL}_2(\F_4)\simeq {\rm PGL}_2(\F_4)\stackrel{\sim}{\lra} ({\rm Aut}_{{\rm alg}}(\mathbb{P}^1_{\F_4}))(\F_4)\hookrightarrow 
{\rm Aut}_{{\rm set}}(\mathbb{P}^1(\F_4))$$
where ${\rm Aut}_{{\rm alg}}(\mathbb{P}^1_{\F_4})$ stands for a group scheme of 
the algebraic automorphisms of $\mathbb{P}^1_{\F_4}$ while 
${\rm Aut}_{{\rm set}}(\mathbb{P}^1(\F_4))$ stands for the set theoretic automorphisms. 
It follows ${\rm Aut}_{{\rm set}}(\mathbb{P}^1(\F_4))\simeq S_5$ since $|\mathbb{P}^1(\F_4)|=5$.  
By counting elements, we can check the image of the above embedding is, indeed, $A_5$. 
Fix a generator $a$ of $\F^\times_4$. 
Let us give all conjugacy classes of 
${\rm SL}_2(\F_4)$ as below:
\begin{itemize}
\item the identity matrix $I_2$;
\item $\begin{pmatrix}
1 & 0 \\
1 & 1
\end{pmatrix}$ the class of order 2;
\item $\begin{pmatrix}
a & 0 \\
0 & a^{-1}
\end{pmatrix}$ the class of order 3;
\item $\begin{pmatrix}
1 & 0 \\
1 & a
\end{pmatrix}$ the class of order 5;
\item $\begin{pmatrix}
1 & 0 \\
1 & a^2
\end{pmatrix}$ the class of order 5.  
\end{itemize} 

\subsection{Representations of $A_5$ to ${\rm GSp}_4(\bF_2)$}\label{repA5}
As in the previous subsection, we fix an isomorphism 
$A_5\simeq {\rm SL}_2(\F_4)$ and let us denote its tautological faithful 
representation by $\ot:A_5\lra {\rm GL}_2(\F_4)$ whose image coincides with ${\rm SL}_2(\F_4)$. 
Since the determinant character of $\ot$ is trivial, $\ot_1:={\rm Sym}^3\ot:A_5\lra {\rm GL}_4(\F_4)$ 
factors through ${\rm GSp}_4(\F_4)$. By direct computation, one can check ${\rm End}_{\ot_1}(\F^{\oplus 4}_4)=\F_4$. 
Hence $\ot_1$ is absolutely irreducible. By Example 5.1.9 of \cite{BPPTVY}, the restriction of $S_5(b)$ 
(see the table in Lemma 5.1.7 of loc.cit.) to $A_5$ gives a representation of $A_5$ to ${\rm GSp}_4(\F_2)$ 
which is equivalent to $\ot_1$ by Brauer-Nesbitt's theorem. 
Notice that ${\rm Im}(\ot_1)$ contains no element of type $(3,3)$ under a fixed isomorphism ${\rm GSp}_4(\F_2)\simeq S_6$. 

On the other hand, we regard $\F_4$ as a two dimensional vector space over $\F_2$. Put $V=\F^{\oplus 2}_4\simeq \F^{\oplus 4}_2$. 
Then a map $\langle\ ,\ \rangle:V\times V\lra \F_2,\ ((x_1,y_1),(x_2,y_2))\mapsto {\rm tr}_{\F_4/\F_2}(x_1x_2+y_1y_2)$ 
yields a non-degenerate alternating form on $V$. Hence we have another representation 
$\ot_2:A_5\lra {\rm SL}_2(\F_4)\lra {\rm GSp}_4(V,\langle \ast,\ast \rangle)\simeq {\rm GSp}_4(\F_2)$. 
It is easy to see that it is irreducible over $\F_2$ but not absolutely irreducible. In fact, over $\F_4$, 
it is included in the endoscopic subgroup $H(\F_4)$ up to conjugacy. 
Notice that ${\rm Im}(\ot_2)$ contains an element of type $(3,3)$.  
Combining the above observation with Lemma 5.1.7 of \cite{BPPTVY}, we have the following:
\begin{prop}\label{ab-irr}Let $G$ be a group and $\br:G\lra {\rm GSp}_4(\F_2)$ be a representation. 
Fix an isomorphism ${\rm GSp}_4(\F_2)\simeq S_6$. 
Suppose that ${\rm Im}(\br)\simeq A_5$ and ${\rm Im}(\br)$ contains no element of type $(3,3)$. 
Then $\br$ is absolutely irreducible and it is equivalent to $\ot_1$ via the isomorphism ${\rm Im}(\br)\simeq A_5$. 
\end{prop}

\subsection{Semi-simplicity}
Let $\F_q$ be a finite field of characteristic 2. 
Let $G$ be a group and $\br:G\lra {\rm GSp}_4(\F_q)$ be a representation.
As in the previous section, by taking the base change to $\F_q$ we have the identification 
${\rm GSp}_4(\F_q)={\rm GSp}(W_{\F_q},\langle \ast,\ast \rangle_{W_{\F_q}})$. 
The dual $\br^\vee$  of $\br$ is naturally identified with $\br$ by using $\langle \ast,\ast \rangle_{W_{\F_q}}$.  
\begin{prop}\label{red}Keep the notation as above. 
There exists a symplectic basis of $W$ such that if $\br$ is reducible, then 
either of the followings holds:
\begin{enumerate}
\item ${\rm Im}(\br)\subset B(\F_q)$;
\item ${\rm Im}(\br) \subset P(\F_q)$ but $\br\not \subset B(\F_q)$;
\item ${\rm Im}(\br) \subset Q(\F_q)$ but $\br\not \subset B(\F_q)$;
\item ${\rm Im}(\br) \subset H(\F_q)$ but neither of parabolic subgroups contains ${\rm Im}(\br)$;
\end{enumerate}
\end{prop}   
\begin{proof}
It follows by using the alternating pairing. 
\end{proof}
\begin{cor}Keep the notation in Proposition \ref{red}. Assume that $\F_q=\F_2$. If ${\rm Im}(\br)$ contains an element of order five, then 
$\br$ is irreducible. 
\end{cor}
\begin{proof}It follows from Lemma \ref{str} and Proposition \ref{red}. 
\end{proof}
Recall the subgroups of $S_6$ whose order divisible by five are given as follows:
$$C_5,D_{10}:=C_2\ltimes C_5,\ F_{20}:=C_4\ltimes C_5,\ A_5,\ S_5,\ A_6,\ S_6.$$
Let us fix an isomorphism ${\rm GSp}_4(\F_2)\simeq S_6$. 
\begin{prop}\label{abs-irr}
Let $G$ be a group and $\br:G\lra {\rm GSp}_4(\F_2)$ be a representation such that 
the order of ${\rm Im}(\br)$ is divisible by five and ${\rm Im}(\br)$ contains no element of 
type $(3,3)$. Then the followings are equivalent 
\begin{enumerate}
\item ${\rm Im}(\br)$ is isomorphic to neither $C_5$ nor $D_{10}$;
\item $\br$ is absolutely irreducible. 
\end{enumerate} 
\end{prop}
\begin{proof}Suppose that ${\rm Im}(\br)$ is isomorphic to either $C_5$ or $D_{10}$. 
In either case, ${\rm Im}(\br)$ contains $C_5$ and let us put $H=\br^{-1}(C_5)$.  
Then $\br|_H$ is reducible over $\F_{2^4}$. Let $\omega_4:H\lra \F^\times_{2^4}$ be a 
component of $\br|_H$ which is of order five. Then it is easy to see that 
$\br|_H\simeq \omega_4\oplus \omega^2_4\oplus \omega^4_4\oplus \omega^8_4$.  
Therefore, the case when ${\rm Im}(\br)\simeq C_5$ is done. 
In the case when ${\rm Im}(\br)\simeq D_{10}$, we have $\br\simeq {\rm Ind}^{G}_{H}\omega_4\oplus {\rm Ind}^{G}_{H}\omega^2_4$. 

Assume to the contrary. If $\br$ is not absolutely irreducible, we see that it is contained in 
$P'(\F_q)$ or $H(\F_q)$ for  
 a finite extension $\F_{q}/\F_2$ where $P'$ is one of three parabolic subgroups in Section \ref{gsp4}. 
Assume ${\rm Im}(\br)\subset P'(\F_q)$. Then an element of order five in ${\rm Im}(\br)$ has to belong to the 
Levi factor of $P'$ which normalized the unipotent radical of $P'$. However, this can not occur by the 
structure of $P'$. 
When ${\rm Im}(\br)\subset H(\F_q)$ but ${\rm Im}(\br)\not \subset P'(\F_q)$ for 
any parabolic subgroup $P'$, by Lemma 5.1.7, Example 5.1.11 of \cite{BPPTVY}, the only possibility is 
${\rm Im}(\br)\simeq A_5$ and $\br$ is obtained from the natural action of $A_5\simeq {\rm SL}_2(\F_4)$ 
acting on $\F^{\oplus 2}_4\simeq \F^{\oplus 4}_2$. Then the image contains an element of type (3,3) but 
it contradicts the assumption.   
\end{proof}
Let $G$ be a group with a subgroup $H$ of index 2. 
For a representation $\rho:H\lra {\rm GL}(V)$ of $H$ 
and a lift $\widetilde{h}$ of the generator $h$ of  $G/H$ to $G$, 
we define the representation $\rho\otimes {}^h\rho$ of $G$ by, for each $x\otimes y\in V$, 
$$(\rho\otimes {}^h\rho)(g)(x\otimes y) 
=\left\{\begin{array}{ll}
 \rho(g)x\otimes\rho(\widetilde{h}^{-1}g \widetilde{h})y & (g\in H) \\
 y\otimes\rho(\widetilde{h}^2)x & (g=\widetilde{h})
 \end{array}\right.
 $$
 whose isomorphism class is independent of the choice of $\widetilde{h}$. For any quadratic extension $M/K$ of fields with 
${\rm Gal}(M/K)=\langle \iota \rangle$ and 
a Galois representation $\rho$ of $G_M$, the representation $\rho\otimes {}^\iota \rho$ of $G_K$ is said to be 
the twisted tensor product of $\rho$ (cf. Section 2.1 of \cite{Berger}).   

Let us fix an isomorphism ${\rm GSp}_4(\F_2)\simeq S_6$. 

\begin{prop}\label{s5Galois}Let $K$ be a totally real field. Let $\br:G_K\lra {\rm GSp}_4(\F_2)$ be an irreducible  
mod 2 Galois representation. Suppose that ${\rm Im}(\br)$ is isomorphic to $S_5$ and ${\rm Im}(\br)$ contains no element of 
type $(3,3)$.   
Assume further that for each complex conjugation $c$ of $G_K$, $\br(c)$ is of type $(2,2)$. 
Then there exist a totally real quadratic extension $M/K$ with 
${\rm Gal}(M/K)=\langle \iota \rangle$  and an irreducible totally odd Galois 
representation $\ot:G_M\lra {\rm GL}_2(\F_4)$ satisfying 
\begin{enumerate}
\item $\br \simeq \ot\otimes {}^\iota \ot$ 
as a representation to ${\rm GL}_4(\F_4)$;
\item ${\rm Im}(\ot)\simeq {\rm SL}_2(\F_4)\simeq A_5$;
\item  for each complex conjugation $c$ of $G_K$, $\ot(c)$ is conjugate to 
$s:=\begin{pmatrix}
0 & 1  \\
1 & 0 
\end{pmatrix}$.
\end{enumerate}
\end{prop}
\begin{proof}
Let us fix an isomorphism ${\rm Im}(\br)\simeq S_5$ and consider the composition $\mu:=\sgn\circ \br:G_K\lra \{\pm 1\}$
where $\sgn:S_5\lra \{\pm 1\}$ stands for the sign character. 
By assumption, for each complex conjugation $c$ of $G_K$, $\mu(c)=1$. 
Hence there exists a totally real quadratic extension $M/K$ corresponding to the kernel of $\nu$ such that 
$\br(G_M)\simeq A_5$. 
Since $A_5\simeq {\rm SL}_2(\F_4)$, 
there is a representation $\ot:G_M\lra {\rm SL}_2(\F_4)$ such that ${\rm Im}(\br|_{G_M})\simeq {\rm Im}(\ot)$. 
Under this isomorphism, by comparing characteristic polynomials of all elements in $S_5$ by using results in 
Section \ref{GSp4vsS6} and Section \ref{A5}, we can easily check 
that $$\br\simeq \ot\otimes{}^\iota \ot$$
by the Brauer-Nesbitt theorem.  
  
Further, by assumption, for each complex conjugation $c$ of $G_K$, 
$\br(c)$ is conjugate to either $J$ or 
$\begin{pmatrix}
s & 0_2  \\
0_2 & s 
\end{pmatrix}$. They are not conjugate in ${\rm GSp}_4(\F_2)$ but conjugate each other in ${\rm GL}_4(\F_2)$. 
The third claim follows from this. 
\end{proof}

\begin{rmk}\label{symp-ortho}
\begin{enumerate}
\item The group $S_5$ can be realized in ${\rm GSp}_4(\F_2)$ as a twisted tensor product of 
a mod 2 two-dimensional representation $\ot$ of 
$A_5$ twisted by $S_5/A_5$ as in 
Theorem \ref{s5Galois}. As in Proposition \ref{abs-irr}, it is absolutely irreducible even after the restriction to 
$A_5$.  
On the other hand, ${\rm Ind}^{S_5}_{A_5}\ot$ is a representation to ${\rm GSp}_4(\F_4)$ which can not be descend to 
${\rm GSp}_4(\F_2)$ since it contains an element of order 3 with the characteristic polynomial $(1+t+t^2)^2$. 
Such representations appear in the context of \cite{Taylor1} regarding abelian surfaces with 
real multiplication by $\Q(\sqrt{5})$.  
\item It is easy to see that 
$\wedge^2 {\rm Ind}^{G_K}_{G_M}\ot\simeq (\ot\otimes{}^\iota \ot)\oplus \textbf{1}\oplus \textbf{1}$. 
\end{enumerate}
\end{rmk}

\begin{prop}\label{a5Galois}Let $K$ be a totally real field. Let $\br:G_K\lra {\rm GSp}_4(\F_2)$ be an irreducible  
mod 2 Galois representation. Suppose that ${\rm Im}(\br)$ is isomorphic to $A_5$ and ${\rm Im}(\br)$ contains no element of 
type $(3,3)$.  
Assume further that for each complex conjugation $c$ of $G_K$, $\br(c)$ is of type $(2,2)$. 
Then there exists  an irreducible totally odd Galois 
representation $\ot:G_K\lra {\rm GL}_2(\F_4)$ satisfying 
\begin{enumerate}
\item $\br \simeq {\rm Sym}^3(\ot)$ as a representation to ${\rm GL}_4(\F_4)$;
\item ${\rm Im}(\ot)\simeq {\rm SL}_2(\F_4)\simeq A_5$;
\item  for each complex conjugation $c$ of $G_K$, $\ot(c)$ is conjugate to 
$s:=\begin{pmatrix}
0 & 1  \\
1 & 0 
\end{pmatrix}$.
\end{enumerate}
\end{prop}
\begin{proof}As in the proof of the previous proposition, we have only to compute the 
characteristic polynomials by using the results in 
Section \ref{GSp4vsS6}, Section \ref{A5}, and Proposition \ref{ab-irr}.  
\end{proof}

\section{Automorphy}
\subsection{Automorphic Galois representations}
Let $F$ be a totally real field. For each place $v$ of $F$, let $F_v$ be the completion of $F$ along $v$.  
In this section we recall basic properties of cuspidal automorphic 
representations of ${\rm GSp}_4(\A_F)$ whose infinite components are 
discrete series representations. We basically follow the notation of Mok's article \cite{Mok} and add more 
necessarily ingredients for our purpose.  

For any place $v$ of $F$, we denote by $W_{F_v}$ the Weil group of $F_v$. 
Let $m_1,m_2,w$ be integers such that $m_1>m_2> 0$ and $m_1+m_2\equiv w+1$ mod 2. 
For the L-parameter $\phi_{(w;m_1,m_2)}:W_\R\lra {\rm GSp}_4(\C)$ defined by 
$$\phi_{(w;m_1,m_2)}(z)=|z|^{-w}\diag\Big(\Big(\frac{z}{\overline{z}}\Big)^{\frac{m_1+m_2}{2}},
\Big(\frac{z}{\overline{z}}\Big)^{\frac{m_1-m_2}{2}},
\Big(\frac{z}{\overline{z}}\Big)^{-\frac{m_1-m_2}{2}},
\Big(\frac{z}{\overline{z}}\Big)^{-\frac{m_1+m_2}{2}}\Big)$$
and 
$$\phi_{(w;m_1,m_2)}(j)=
\left(\begin{array}{cc}
0_2 & s \\
(-1)^w s & 0_2
\end{array}
\right)
,\ s=\left(\begin{array}{cc}
0 & 1 \\
1 & 0
\end{array}
\right).$$ 
By local Langlands correspondence the archimedean L-packet $\Pi(\phi_{(w;m_1,m_2)})$ corresponding to $\phi_{(w;m_1,m_2)}$ 
consists of two elements $\{\pi^H_{(w;m_1,m_2)},\ \pi^W_{(w;m_1,m_2)}\}$ whose 
central characters both satisfy $z\mapsto z^{-w}$ for $z\in \R^\times_{>0}$.  
These are essentially tempered unitary representations of ${\rm GSp}_4(\R)$ and tempered exactly when $w=0$. 
Since $m_2>0$, the representation $\pi^H_{(w;m_1,m_2)}$ is called a discrete series representation of 
minimal $K$-type $\uk=(k_1,k_2):=(m_1+1,m_2+2)$ which corresponds to an algebraic representation 
$$V_{\uk}:={\rm Sym}^{k_1-k_2}{\rm St}_2\otimes {\rm det}^{k_2}{\rm St}_2=
{\rm Sym}^{m_1-m_2-1}{\rm St}_2\otimes {\rm det}^{m_2+2}{\rm St}_2$$ of $K_\C={\rm GL}_2(\C)$. 
Here $K$ is the maximal compact subgroup of ${\rm Sp}_4(\R)$. 
The representation $\pi^H_{(w;m_1,m_2)}$ is  called  a  discrete series 
representation of minimal $K$-type 
$$V_{(m_1+1,-m_2)}={\rm Sym}^{m_1+m_2+1}{\rm St}_2\otimes {\rm det}^{-m_2}{\rm St}_2.$$ 

Fix an integer $w$.  
Let $\pi=\otimes'_{v}\pi_v$ be a cuspidal automorphic representation of ${\rm GSp}_4(\A_F)$ such that 
for each infinite place $v$, $\pi_v$ has L-parameter $\varphi_{(w;m_{1,v},m_{2,v})}$ with 
the parity condition $m_{1,v}+m_{2,v}\equiv w+1$ mod 2 and $m_2>0$. Let ${\rm Ram}(\pi)$ be the set of all 
finite places of which $\pi_v$ is ramified. 
Thanks to \cite{Mok} with \cite{GeeT} we can attach $\pi$ with Galois representations:
\begin{thm}\label{gal}(cf. Theorem 3.1 and Remark 3.3, Theorem 1.1 of \cite{Mok}) 
Assume that $\pi$ is neither CAP nor endoscopic. 
For each prime $p$ and $\iota_p:\bQ_p\stackrel{\sim}{\lra} \C$ 
there exists a continuous, semisimple Galois representation 
$\rho_{\pi,\iota_p}:G_F\lra {\rm GSp_4}(\bQ_p)$ such that 
\begin{enumerate}
\item $\nu\circ \rho_{\pi,\iota_p}(c_\infty)=-1$ for any complex conjugation $c_\infty$ in $G_F$ where 
$\nu$ is the similitude character of ${\rm GSp}_4$. 

\item $\rho_{\pi,\iota_p}$ is unramified for all finite places which do not belong to ${\rm Ram}(\pi)\cup\{v|p\}$;

\item for each finite place $v$ of $F$ not lying over $p$, the local-global compatibility holds:
$${\rm WD}(\rho_{\pi,\iota_p}|_{G_{F_v}})^{F-{\rm ss}}\simeq {\rm rec}^{{\rm GT}}_v(\pi_v\otimes |\nu|^{-\frac{3}{2}})$$
with respect to $\iota_p$ where ${\rm rec}^{{\rm GT}}_v$ stands for the local Langlands correspondence 
constructed by Gan-Takeda \cite{GT};

\item  for each finite place $v$ of $F$ lying over $p$,  $\rho_{\pi,\iota_p}|_{G_{F_v}}$ is crystalline and the local-global compatibility also holds up to semi-simplification. 

\item for each $v|p$ and an embedding $\sigma:F_v\hookrightarrow \bQ_p$, there is a unique embedding 
$v_\sigma:F\hookrightarrow \C$ such that $\iota_p\circ \sigma|_F=v_\sigma$. 
Then the representation $\rho_{\pi,\iota_p}|_{G_{F_v}}$ is Hodge-Tate of weights 
$$HT_\sigma(\rho_{\pi,\iota_p}|_{G_{F_v}})=\{\delta_{v_\sigma},\delta_{v_\sigma}+m_{2,v_\sigma},
\delta_{v_\sigma}+m_{1,v_\sigma},\delta_{v_\sigma}+m_{2,v_\sigma}+m_{1,v_\sigma} \}$$
where $\delta_{v_\sigma}=\ds\frac{1}{2}(w+3-m_{1,v_\sigma}-m_{2,v_\sigma})$.  
\end{enumerate}
\end{thm}
\begin{dfn}\label{auto-gal} 
\begin{enumerate}
\item
Let $\rho:G_F\lra {\rm GSp}_4(\bQ_p)$ be an irreducible $p$-adic Galois representation. 
We say $\rho$ is automorphic if there exists a cuspidal automorphic representation $\pi$ of 
${\rm GSp}_4(\A_F)$ with $\pi_v$ a discrete series representation for any $v|\infty$  such that 
$\rho\simeq\rho_{\pi,\iota_p}$ as a representation to ${\rm GL}_4(\bQ_p)$. 
By definition, if $\rho$ is automorphic, then it is totally odd.  
\item  Let $\br:G_F\lra {\rm GSp}_4(\bF_p)$ be an irreducible mod $p$  Galois representation. 
We say $\br$ is automorphic if there exists a cuspidal automorphic representation $\pi$ of 
${\rm GSp}_4(\A_F)$ with $\pi_v$ a discrete series representation for any $v|\infty$  such that 
$\br\simeq \br_{\pi,\iota_p}$ as a representation which takes the values in ${\rm GL}_4(\bF_p)$. 
\end{enumerate}
\end{dfn} 
\begin{rmk}\label{switch}
\begin{enumerate}
\item For each holomorphic Hilbert-Siegel Hecke eigen cusp form $h$ over $F$ on $\mathcal{H}^d_2$, one can 
associate a cuspidal automorphic form on ${\rm GSp}_4(\A_F)$ and vise versa. We often identify these two forms. 
\item Let $h$ be a holomorphic Hilbert-Siegel Hecke eigen cusp form on $\mathcal{H}^d_2$ of parallel weight $3$ where $d=[F:\Q]$. 
Let $\pi_h$ be the corresponding cuspidal automorphic representation of ${\rm GSp}_4(\A_F)$. 
Then for each infinite place $v$, the local Langlands parameter at $v$ is given by $\phi_{(w;2,1)}$ for some $w\in \Z$. 
Conversely, if cuspidal automorphic representation $\pi$ of ${\rm GSp}_4(\A_F)$ is neither CAP nor endoscopic, one can associate such a form $h$ by using 
\cite{Wei3} for $F=\Q$ and \cite{GeeT} in general. 
Note that the results in \cite{GeeT} are conditional on the trace formula 
$($see the second paragraph in p.472 of \cite{GeeT}$)$.   
\end{enumerate} 
\end{rmk}

\subsection{Paritious Hilbert modular forms and Jacquet-Langlands correspondence}\label{JL} 
We refer \cite{Taylor1} and Section 3 of \cite{Kisin}.  
In this section, $p$ is any rational prime but we always remind the readers applications to $p=2$. 
Let $F$ be a totally real field of even degree $g$. 
For each finite place $v$ of $F$, let $F_v$ be 
the completion of $F$ along $v$, $\O_v$ its integer ring, $\varpi_v$ a uniformizer of $F_v$, and $\F_v$ the 
residue field of $F_v$.   
Let $D$ be a quaternion algebra with center $F$ which is 
ramified exactly at all the infinite places of $F$ and $\O_D$ be the ring of all integral quaternions of $D$. 
For each finite place $v$ of $F$, we fix an 
isomorphism $\iota_v:D_v:=D\otimes_FF_v\simeq {\rm GL}_2(F_v)$. 
We view $D^\times$ as an algebraic group over $F$ so that for any $F$-algebra $A$, $D^\times(A)$ outputs 
$(D\otimes _FA)^\times$ and similarly as an algebraic group scheme over $\O_F$ such that 
$D^\times (R)=(\O_D\otimes_{\O_F}R)^\times$ for any $\O_F$-algebra $R$.  

Let $K$ be a finite extension of $\Q_p$ contained in $\bQ_p$ with residue field $k$ and $\mathcal{O}$ the ring of integers, and 
assume that $K$ contains the images of all embeddings $F\hookrightarrow \bQ_p$. 

For each finite place $v$ lying over $p$ of $F$, let $\tau_v$ be  a smooth representation of ${\rm GL}_2(\mathcal{O}_{v})$ on a finite free 
$\O$ module $W_{\tau_v}$. We also view it as a representation of $D_v$ via $\iota_v$. 
Put $\tau:=\ds\otimes_{v|p}\otimes_{\sigma_v\in {\rm Hom}(F_v,\bQ_p)}{}^{\sigma_v}\tau_v$ which is a representation of ${\rm GL}_2(\mathcal{O}_{p})$ 
acting on $W_\tau:=\otimes_{v|p}\otimes_{\sigma_v\in {\rm Hom}(F_v,\bQ_p)}{}^{\sigma_v}W_{\tau_v}$. Suppose 
$\psi:F^\times\bs (\A^\infty_F)^\times\lra \mathcal{O}^\times$ is a continuous character so that 
for each $v|p$, $Z_{D^\times}(\mathcal{O}_{v})\simeq \mathcal{O}^\times_{v}$ acts on 
$W_{\tau_v}$ by $\psi^{-1}|_{\mathcal{O}^\times_{F_v}}$ where $Z_{D^\times}\simeq GL_1$ be the center of $D^\times$ 
as a group scheme over $\O_F$. Note that we put the discrete topology on $\mathcal{O}^\times$ and 
therefore such a character is necessarily of finite order. Let $U=\prod_v U_v$ be a compact open subgroup 
of $D^\times(\A^\infty_{F})\simeq {\rm GL}_2(\A^\infty_{F})$ such that $U_v\subset D^\times({O}_{F_v})$ 
for each finite place $v$ of $F$. 
Put $U_p:=\prod_{v|p}U_v$ and $U^{(p)}=\prod_{v\nmid p}U_v$. 
For any local $\mathcal{O}$-algebra $A$ put $W_{\tau,A}:=W_\tau\otimes_{\mathcal{O}} A$. Let 
$\Sigma$ be a finite set of finite places of $F$. 
For each $v\in \Sigma$, let $\chi_v:U_v\lra A^\times$ be a quasi character. Define $\chi_{\Sigma}:U\lra A^\times$ whose local component is $\chi_v$ if $v\in \Sigma$, the trivial representation otherwise. 

\begin{dfn}\label{dfn-2AMF}$(p$-adic algebraic quaternionic forms$)$ 
Let $S_{\tau,\psi}(U,A)$ denote the space of the functions $f:D^\times\bs D^\times(\A^\infty_{F})
\lra W_{\tau,A}$ such that 
\begin{itemize}
\item $f(gu)=\tau(u_p)^{-1}f(g)$ for $u=(u^{(p)},u_p)\in U=U^{(p)}\times U_p$ and $g\in D^\times(\A^\infty_{F})$;
\item $f(zg)=\psi(z)f(g)$ for $z\in Z_{D^\times}(\A^\infty_F)$ and $g\in D^\times(\A^\infty_{F})$.
\end{itemize}
Similarly, 
let $S_{\tau,\psi,\chi_{\Sigma}}(U,A)$ denote the space of the functions $f:D^\times\bs D^\times(\A^\infty_{F})\lra W_{\tau,A}$ such that 
\begin{itemize}
\item $f(gu)=\chi^{-1}_{\Sigma}(u)\tau(u_p)^{-1}f(g)$ for $u=(u^{(p)},u_p)\in U=U^{(p)}\times U_p$ and 
$g\in D^\times(\A^\infty_{F})$;
\item $f(zg)=\psi(z)f(g)$ for $z\in Z_{D^\times}(\A^\infty_F)$ and $g\in D^\times(\A^\infty_{F})$.
\end{itemize}
We say a function belongs to these spaces a $p$-adic algebraic quaternionic form. 
\end{dfn} 

Let $S$ be a finite set of finite places of $F$ containing 
all places $v\nmid p$ such that $U_v\neq D^\times(\mathcal{O}_{v})$.  
We define the (formal) Hecke algebra 
\begin{equation}\label{fha1}
\mathbb{T}^S_A:=A[T_{v},S_v]_{v\not\in S\cup\{v|p\}}
\end{equation}
where $T_{v}=[D^\times(\mathcal{O}_{v})\iota^{-1}_{v}(\diag(\varpi_v,1))D^\times(\mathcal{O}_{v})],
S_v=[D^\times(\mathcal{O}_{v})\iota^{-1}_{v}(\diag(\varpi_v,\varpi_v))D^\times(\mathcal{O}_{v})]$ 
are usual Hecke operators. 
It is easy to see that both of $S_{\tau,\psi}(U,A)$ and $S_{\tau,\psi,\chi_{\Sigma}}(U,A)$ have a natural action of $\mathbb{T}^S_A$ (cf. Definition 2.2 of \cite{Dem}).

Let $U=U^{(p)}\times U_p$ be as above. As explained in Section 1 \cite{Taylor1}, if we take the 
double coset decomposition 
${\rm GL}_2(\A^\infty_F)=\coprod_i D^\times t_i UZ_{D^\times}(\A^\infty_F)$, then 
$$S_{\tau,\psi}(U,A)\simeq \bigoplus_i W^{(UZ_{D^\times}(\A^\infty_F)\cap t^{-1}_iD^\times t_i)/F^\times}_\tau.$$
The group $(UZ_{D^\times}(\A^\infty_F)\cap t^{-1}_iD^\times t_i)/F^\times$ is trivial for all $t_i$ when 
$U$ is sufficiently small (see p.623 of \cite{Kisin}). Henceforth, we keep this condition until the end of 
the section. It follows from this that the functor $W_\tau \mapsto S_{\tau,\psi}(U,A)$ is exact 
(cf. Lemma 3.1.4 of \cite{Kisin-finite}).

Fix an isomorphism $\iota:\bQ_p\simeq \C$. Let 
$$S_{\tau,\psi}(U_p,A):=\varinjlim_{U^{(p)}}S_{\tau,\psi}(U^{(p)}\times U_p,A)$$
where $U^{(p)}$ tends to be small. 
For each $(\underline{k},\underline{w})=((k_\sigma)_\sigma,(w_\sigma)_\sigma)\in \Z^{{\rm Hom}(F,\bQ_p)}_{>1}\times \Z^{{\rm Hom}(F,\bQ_p)}$ such that 
$k_\sigma+2w_\sigma$ is independent of $\sigma$. This independence is called the parity condition and it is 
necessary to arithmetic structure on the space of Hilbert modular forms.  
Let $\psi_\C:F^\times\bs \A^\times_F\lra \C^\times$ be 
the character defined by 
\begin{equation}\label{psi}\psi_\C(z)=\iota(N(z_p)^{\delta-1}\psi(z^\infty))N(z_\infty)^{1-\delta}
\end{equation} for 
$z=(z_p,z^{(p)},z_\infty)\in  \A^\times_F$ where the symbol $N$ stands for the norm. 
For each $\sigma\in {\rm Hom}(F,\bQ_p)$, there exist a unique pair of $v|p$ and an embedding $\sigma_v:F_v\lra \bQ_p$ such that $\sigma_v|_F=\sigma$. There we can rewrite $(\sigma)_{\sigma \in {\rm Hom}(F,\bQ_p)}=(\sigma_v)_
{v|p,\ \sigma_v\in {\rm Hom}(F_v,\bQ_p)}$.  
Let us define the representation $\tau_{(\underline{k},\underline{w}),A}$ of 
${\rm GL}_2(\O_p)=\prod_{v|p}{\rm GL}_2(\O_v)$ by 
\begin{equation}\label{alg-rep}
\tau_{(\underline{k},\underline{w}),A}=\bigotimes_{v|p}\bigotimes_{\sigma_v\in {\rm Hom}(F,\bQ_p)}
{\rm Sym}^{k_{\sigma_v}-2}{\rm St}_2(A)\otimes {\rm det}^{w_{\sigma_v}}A
\end{equation}
where ${\rm St}_2$ is the standard representation of dimension two. We often drop the subscript $A$ from 
$\tau_{(\underline{k},\underline{w}),A}$ which never causes any confusion.  
Notice that $\tau_{(\underline{k},\underline{w}),\O}\otimes_{\O,\iota}\C$ is the algebraic representation of ${\rm GL}_2(\C)$ of the 
highest weight $(\underline{k},\underline{w})$ so that the center acts by $z\mapsto z^{\delta-1},\ \delta=k_\sigma+2w_\sigma-1$ for 
$z\in \C^\times$. We write 
\begin{equation}\label{n-cl}
S_{\underline{k},\underline{w},\psi}(U_p,A):=
S_{\tau_{(\underline{k},\underline{w})},\psi}(U_p,\O)
\end{equation}
for simplicity. 

By Lemma 1.3-2 of \cite{Taylor1}, we have an isomorphism of $D^\times(\A^{p,\infty})$-modules 
\begin{equation}\label{classical}
(S_{\underline{k},\underline{w},\psi}(U_p,\O)/
S^{{\rm triv}}_{\underline{k},\underline{w},\psi}(U_p,\O))\otimes_{\O,\iota}\C\simeq \bigoplus_{\pi}\pi^{\infty,p}\otimes \pi^{U_p}_p 
\end{equation}
where $\pi$ turns over regular algebraic cuspidal automorphic representations of ${\rm GL}_2(\A_F)$ such that 
$\pi$ has the central character $\psi_\C$ and $S^{{\rm triv}}_{\underline{k},\underline{w},\psi}(U_p,\O)$ is 
zero unless $(\underline{k},\underline{w})=(2\cdot\textbf{1},w\cdot\textbf{1})$ with $w\in\Z$ and $\textbf{1}=(1,\ldots,1)\in \Z^{{\rm Hom}(F,\bQ_p)}$, in which case let it denote the subspace of 
$S_{\tau_{(\underline{k},\underline{w})},\psi}(U_p,\O)$ consisting of functions which factor through the reduced norm of $D^\times(\A_F)$. 

For each $\tau_{(\underline{k},\underline{w}),\C}$, one can also consider the space of Hilbert modular cusp forms on ${\rm GL}_2(\A_F)$ of level $U$ and of weight $(\underline{k},\underline{w})$ and 
their geometric counterparts (see Section 1.5 of \cite{Di}).  
Thanks to many contributors (see \cite{Jarvis} and the reference there), for each (adelic) Hilbert modular 
Hecke eigen cusp form $f$ of weight $(\underline{k},\underline{w})$ and of level $U$, one can attach 
an irreducible $p$-adic Galois representation $\rho_{f,\iota_p}:G_F\lra {\rm GL}_2(\bQ_p)$ for any rational prime $p$ 
and a fixed isomorphism $\iota_p:\bQ_p\simeq \C$. 
The construction also works for $\underline{k}\ge \textbf {1}$ (in the lexicographic order). In particular, in the case when $\underline{k}=\textbf{1}$ (parallel weigh one), 
its image is finite and it gives rise to an Artin representation $\rho_f:G_F\lra {\rm GL}_2(\C)$ (see \cite{RT}). 
Taking a suitable integral lattice, we consider the reduction $\br_{f,p}:G_F\lra {\rm GL}_2(\bF_p)$ of 
$\rho_{f,p}$ or $\rho_f$ modulo 
the maximal ideal of $\bZ_p$. 

Recall that we have fixed $\bQ_p\simeq \C$ and under this isomorphism we identify ${\rm Hom}(K,\bQ_p)$ with 
${\rm Hom}(K,\C)$ for any number field $K$. 
If $M$ is a totally real quadratic extension of a totally real field $F$. Put $d=[F:\Q]$ and then $m=[M:\Q]=2d$. 
For each embedding $\sigma_i\in {\rm Hom}(F,\bQ_p)\ (1\le i\le d)$, let $\sigma^{(1)}_i,\ \sigma^{(2)}_i$ 
be the extensions of $\sigma_i$ to $M$ with any ordering for $\sigma^{(1)}_i,\ \sigma^{(2)}_i$.  Put 
$$S^{(1)}_M=\{\sigma^{(1)}_i\ |\ 1\le i\le d\},\ S^{(2)}_M=\{\sigma^{(2)}_i\ |\ 1\le i\le d\}.$$
According to this notation, we rewrite $(\underline{k},\underline{w})=((k_\sigma)_\sigma,(w_\sigma)_\sigma)\in \Z^{{\rm Hom}(M,\bQ_p)}_{\ge 1}\times \Z^{{\rm Hom}(M,\bQ_p)}$ as 
$$(\underline{k},\underline{w})=((\underline{k}^{(1)},\underline{k}^{(2)}),
(\underline{w}^{(1)},\underline{w}^{(2)}))\in 
(\Z^d_{\ge 1})^2\times (\Z^d)^2.$$
Put $\textbf{1}_m=(1,\ldots,1)\in \Z^m$ and $\textbf{1}_d=(1,\ldots,1)\in \Z^d$.  
 
The following result is a balk of this paper:
\begin{prop}\label{cong}Put $p=2$. 
Assume that $M/F$ is a totally real quadratic extension of a totally real field $F$. Put $d=[F:\Q]$ and $m=[M:\Q]=2d$. 
Let $f$ be a Hilbert modular Hecke eigen cusp form on ${\rm GL}_2(\A_M)$ of parallel weight one such that 
$\br_{f,2}$ is irreducible. There exists a Hilbert modular Hecke eigen cusp form $g$ on ${\rm GL}_2(\A_M)$ such that 
\begin{enumerate}
\item $\br_{f,2}\simeq \br_{g,2}\otimes \psi$ for some continuous character $\psi:G_M\lra \bF^\times_2$; 
\item the character corresponding to the central character of $g$ under (\ref{psi}) 
is trivial; 
\item $g$ is of weight $(\underline{k},\underline{w})$ with 
$\underline{k}=2\cdot \textbf{1}_m  ($parallel weight 2$)$ for some $\underline{w}\in \Z^m$. 
\end{enumerate} 
Further $g$ can be congruent modulo a prime lying over 2 to a Hilbert modular Hecke eigen cusp form of weight $(\underline{k},\underline{w}')$ with 
 $$\underline{k}=(2\cdot \textbf{1}_d,4\cdot \textbf{1}_d)$$ for some $\underline{w}'\in \Z^m$.      
\end{prop}
\begin{proof}Let $U$ be a sufficiently small open compact subgroup of ${\rm GL}_2(\A_M)$ which fixes $f$. 
Fix $\underline{w}=w\cdot\textbf{1}_m$ for some $w\in \Z$ in the weight of $f$.  
Let $K$ be a sufficiently large finite extension of $\Q_2$ in $\bQ_2$ with the integer ring $\O$ such that 
$K$ contains all Hecke eigen values of $f$ for $\mathbb{T}^S_{\O}$. Let $\F$ be the residue field of $K$. 
As in \cite{Di}, we can view $f$ a geometric Hilbert modular form over $\O$ via a classical Hilbert modular form 
associated to $f$ and 
consider its base change $\overline{f}$ to $\F$. By multiplying a high power of Hasse invariant, we 
have another form $\overline{g}_1$ of weight $(\underline{k},\underline{w})$ with 
$\underline{k}=n\cdot \textbf{1}_m,\ n\gg 0$ such that  
\begin{enumerate}
\item $\overline{g}_1$ is a liftable to 
a geometric Hilbert modular Hecke eigen form $g_1$ over $\O$ by enlarging $\O$ if necessary;
\item $\br_{g_1,2}\simeq \br_{f,2}$.   
\end{enumerate}
Let us get $g_1$ back to the classical Hilbert modular Hecke eigen cusp form and by Jacquet-Langlands correspondence, 
we have a $2$-adic algebraic quaternionic Hecke eigen form $h_1$ in $S_{\tau_{(\underline{k},\underline{w})},\psi}(U,\O)$ corresponding to $g_1$ 
where $\psi:M^\times\bs \A^{\infty}_M\lra \O^\times$ be the finite character corresponding to the central 
character of $g_1$ under (\ref{psi}). For each finite place $v$ of $M$ lying of over $2$, 
let $U_{1,v}$ be the subgroup $U_v$ consisting of all elements congruent to 
$\begin{pmatrix}
1 & \ast  \\
0 & 1 
\end{pmatrix}$. Put $U_{1,2}:=\prod_{v|2}U_{1,v}$. By definition, $\overline{h}_1\in 
S_{\tau_{(\underline{k},\underline{w})},\psi}(U^{(2)}\times U_{1,2},\F)$. 
Since $U_{1,2}$ acts on $\tau_{(\underline{k},\underline{w}),\F}$ unipotently, 
it can be written as an successive extension of the trivial representation of  $U_{1,2}$. 
This successive extension commutes with Hecke actions and one can find 
a Hecke eigen form $\overline{h}_2\in S_{\tau_{(2\cdot \textbf{1},\underline{w})},\psi}(U^{(2)}\times U_{1,2},\F)$. 
Further, notice that $(\F^\times)^2=\F^\times$ and by twisting, we may assume that the reduction 
$\overline{\psi}$ of $\psi$ is trivial. 
Since $U$ is taken to be sufficiently large, there exists a lift 
$h_2\in S_{\tau_{(2\cdot \textbf{1},\underline{w})},\psi^{{\rm triv}}}(U^{(2)}\times U_{1,2},\O)$ of $\overline{h}_2$ 
where $\psi^{{\rm triv}}$ stands for the trivial character.  
By using Lemma 6.11 of \cite{DS}, $h_2$ can be a Hecke eigenform by multiplicity one. 
Further by Theorem 1.3-4 of \cite{Taylor1}, $h_2$ is non-trivial since $\ot$ is absolutely irreducible. 
The desired form $g$ is obtained as the image of $h_2$ under the Jacquet-Langlands correspondence. 

Let us keep the form $\overline{h}_2\in S_{\tau_{(2\cdot \textbf{1},\underline{w})},\overline{\psi}^{{\rm triv}}}(U^{(2)}\times U_{1,2},\F)$. For each finite place $v$ of $M$ lying over 2 and embedding 
$\sigma_v:M_v\lra \bQ_2$ which induces an embedding $\F_v\hookrightarrow \F$, let us observe the following exact sequence of ${\rm GL}_2(\F_v)$-modules 
\begin{equation}\label{ext}
0\lra {}^{\sigma_v}({\rm St}_2(\F))^{(2)}\lra {}^{\sigma_v}{\rm Sym}^2{\rm St}_2(\F)\lra {}^{\sigma_v}{\rm det}\F\lra 0.
\end{equation} 
Here the superscript ``$\sigma_v$" means the twisted representation induced by the embedding $\F_v\hookrightarrow \F$. 
We apply this exact sequence to $S^{(2)}_M$. 
Since ${}^{\sigma_v}{\rm det}:Z_D(\O_v)\lra \F^\times$ can be viewed as a square of a character, by twisting, 
we have a Hecke eigen lift to  
$\overline{h}_3\in S_{\tau_{(2\cdot \textbf{1}_d,4\cdot \textbf{1}_d,\underline{w}')},{\rm det}^{-1}}(U^{(2)}\times U_{1,2},\F)$. Twisting it back to $\overline{h}_4\in S_{\tau_{(2\cdot \textbf{1}_d,4\cdot \textbf{1}_d,\underline{w}')},\overline{\psi}^{{\rm triv}}}(U^{(2)}\times U_{1,2},\F)$.
We have a lift $g$ of $\overline{h}_4$ which has the desired properties.   
\end{proof}

\subsection{Non-paritious Hilbert modular forms}
We refer \cite{DLP}, \cite{P} for non-paritious Hilbert modular forms.   
Keep the notation in the previous section. 
We can also consider 
Hilbert modular forms on ${\rm GL}_2(\A_M)$ of non-paritious weight. 
However we will immediately encounter the lack of some arithmetic structures for such forms. 
In fact, the automorphic representation associated to a non-paritious Hilbert cusp form is non-algebraic and 
its Hecke field is not finite over $\Q$. This explains that we can not construct any $p$-adic Galois representations 
for such a form which takes the values in ${\rm GL}_2(\bQ_p)$ for any $p$. However, 
if we lift the cuspidal automorphic representation to a suitable group by some Langlands functoriality, 
it can be algebraic, hence we have a Galois representation. 
In \cite{P} with his related works, 
Patrikis has studies there objects throughly. We will use some results in \cite{DLP} which is related to  
\cite{P} and it is rather suitable in our purpose. 
 
Let us recall our setting. 
Let $M/F$ be a real quadratic extension of a totally real field $F$ with the Galois group 
${\rm Gal}(M/F)=\langle \iota \rangle$. Put $d=[F:\Q]$ and $m=[M:\Q]=2d$. 
Put $H={\rm Ker}({\rm Nr})={\rm Ker}(D^\times\stackrel{{\rm Nr}}{\lra}GL_1)$ where ${\rm Nr}$ stands 
for the reduced norm. Clearly, $H(\A^{\infty}_{M})\simeq {\rm SL}_2(\A_M)$.   
For each non-paritious weight $\uk\in \Z^m_{> 1}$ and an open compact subgroup $U$ of 
$H(\A^{\infty}_{M})$ such that $U_v\subset H(\O_v)$ for each finite place $v$ of $M$, we can also consider 
the space of $p$-adic algebraic quaternionic forms on $H(\A_M)$ which we denote it by $S_{\underline{k},\psi}(U_p,\O)$ 
where the weight corresponds to $\tau_{(\uk,\underline{0}),\O}|_{U_2\cap H(\A^\infty_M)}$ with 
$\underline{0}=(0,\ldots, 0)\in \Z^m$. 
We define the (formal) Hecke algebra for $H$ by 
\begin{equation}\label{fha2}
\mathbb{T}^{S,H}_A:=A[T_{2,v}]_{v\not\in S\cup\{v|p\}}
\end{equation}
where $T_{2,v}=[H(\mathcal{O}_{v})\iota^{-1}_{v}(\diag(\varpi_v,\varpi^{-1}_v))H(\mathcal{O}_{v})]$ and 
$S$ is as in (\ref{fha1}).  
As in the case of paritious weight, if $U$ is sufficiently small, then 
the reduction map $S_{\underline{k},\psi}(U_p,\O)\lra S_{\underline{k},\psi}(U_p,\F)$ 
is surjective. 

Let us act the $d$-th symmetric group $S_d$ diagonally on $\Z^m=\Z^d\times \Z^d$. 
Similarly, $(\sigma_1,\ldots,\sigma_d)\in (S_2)^d$ acts on $\Z^d\times \Z^d$ by 
$$((x^{(1)}_1,\ldots,x^{(1)}_d),(x^{(2)}_1,\ldots,x^{(2)}_d))
\mapsto ((x^{(\sigma_1(1))}_1,\ldots,x^{(\sigma_d(1))}_d),(x^{(\sigma_1(2))}_1,\ldots,x^{(\sigma_d(2))}_d)).$$
Then these two actions induce the action of $(S_2)^d\rtimes \Delta S_d$ on $\Z^d\times \Z^d$ where each $\tau\in S_d$ acts on $(S_2)^d$ by 
$(\sigma_1,\ldots,\sigma_d)\mapsto (\sigma_{\tau(1)},\ldots,\sigma_{\tau(d)})$. 
Then we will prove the following: 
\begin{prop}\label{non-pari}
Put $p=2$. Keep the notation as above. 
Let $f$ be a Hilbert modular Hecke eigen cusp form on ${\rm GL}_2(\A_M)$ of parallel weight one such that 
$\br_{f,2}$ is irreducible. There exists a Hilbert modular Hecke eigen cusp form $g$ on ${\rm GL}_2(\A_M)$ such that 
\begin{enumerate}
\item the character corresponds to the central character of $g$ under (\ref{psi}); 
\item $g$ is of weight $(\uk,\uw)$ with 
$\uk=(2\cdot \textbf{1}_d,3\cdot \textbf{1}_d)\in \Z^{S^{(1)}_M}\times \Z^{S^{(2)}_M}$ for some $\underline{w}\in \ds\frac{1}{2}\Z^m$. 
\item $\br_{f,2}\otimes {}^\iota\br_{f,2}\simeq \br_{\Pi,2}\otimes \psi$ for some continuous character $\psi:G_M\lra \bF^\times_2$ where $\Pi$ is the automorphic induction of $\pi_g$ to ${\rm GL}_4(\A_F)$ and 
$\br_{\Pi,2}$ is the reduction of $\rho_{\Pi,\iota_2}:G_F\lra {\rm GL}_4(\bQ_2)$ 
constructed in Theorem 5.1 of \cite{DLP}.
\end{enumerate}
\end{prop}
\begin{proof}As in the proof of Proposition \ref{cong}, 
there exists $\overline{h}\in S_{2\cdot \textbf{1},\underline{w},\overline{\psi}^{{\rm triv}}}(U^{(2)}\times U_{1,2},\F)$ for some $w\in \Z^m$ such that the Hecke eigen system of $\overline{h}$ coincides with one of 
the original $f$ up to twist where $U^{(2)}\times U_{1,2}\subset D^\times(\A^\infty_M)$ is in the proof there.  
Let us restrict $\overline{h}$ to $(U^{(2)}\times U_{1,2})\cap H(\A^\infty_M)$ and denote it by 
$\overline{h}$ again. For each finite place $v\not\in S\cup\{v|2\}$ of $M$, let $\overline{a}_{T_{2,v}}(\overline{h})$ 
be the eigenvalue of $T_{2,v}$ for $\overline{h}$ and $a_{T_{v}}(f)$ for $f$. Then it is easy to see 
that $\overline{a}_{T_{2,v}}(\overline{h})=a^2_{T_v}(f)\ {\rm mod}\ m_{\O}$. 
For each $\sigma\in S^{(2)}_M$, there exist unique finite place $v$ of $M$ lying over 2 and am embedding 
$\sigma_v:M_v\lra \bQ_2$ such that $\sigma_v|_{M}=\sigma$. As in (\ref{ext}) let us observe 
the following exact sequence of ${\rm GL}_2(\F_v)$-modules 
\begin{equation}\label{ext}
0\lra {}^{\sigma_v}({\rm St}_2(\F))^{(2)}\lra {}^{\sigma_v}{\rm Sym}^2{\rm St}_2(\F)\lra {}^{\sigma_v}{\rm det}\F\lra 0.
\end{equation}
Since this is non-split as a $U_{1,2}$-module, there exists a Hecke eigen form 
$\overline{h}_2\in S_{\tau,\overline{\psi}^{{\rm triv}}}((U^{(2)}\times U_{1,2})\cap H(\A^\infty_M),\F)$ 
for $\tau=\ds\bigotimes_{\sigma\in S^{(2)}_M}{\rm Sym}^{1}{\rm St}^{(2)}_2(\F)$ whose Hecke system for 
$\mathbb{T}^{S,H}_{\F}$ is same 
as one of $\overline{h}$. 
Notice that the Frobenius twist induces a permutation on ${\rm Hom}(\F_v,\F)$ and the central character is 
always trivial since the center of $H$ is $Z_H=\mu_2$ the group scheme of order 2.  
Then it is easy to see that there exists a subset 
$X\subset S^{(1)}_M\coprod S^{(2)}_M={\rm Hom}(M,\bQ_2)$ with $|X|=d$ such that 
$\tau\simeq \bigotimes_{\sigma\in X}{\rm Sym}^{1}{\rm St}_2(\F)$. 
Let $X^c$ be the complement of $X$ in ${\rm Hom}(M,\bQ_2)$. 
Then there exists an element $\alpha\in (S_2)^d\rtimes S_d$ such that 
$$\alpha((2\cdot\textbf{1}_d,3\cdot\textbf{1}_d)_{S^{(1)}_M\times S^{(2)}_M})=
(2\cdot\textbf{1}_d,3\cdot\textbf{1}_d)_{X^c\times X}.$$
Hence we may assume that $\overline{h}_2\in 
S_{(2\cdot\textbf{1}_d,3\cdot\textbf{1}_d),\overline{\psi}^{{\rm triv}}}((U^{(2)}\times U_{1,2})\cap H(\A^\infty_M),\F)$. 
Then, by Lemma 6.11 of \cite{DS}, there exist a Hecke eigen form  
$$h_2\in 
S_{(2\cdot\textbf{1}_d,3\cdot\textbf{1}_d),\overline{\psi}^{{\rm triv}}}((U^{(2)}\times U_{1,2})\cap H(\A^\infty_M),\O)$$
whose reduction of the Hecke eigen system coincides with one of $\overline{h}_2$. 
Though we can not conclude  that $\overline{h}_2$ coincides with the reduction of $h_2$ due to 
the lack of the multiplicity one, it is enough to observe only Hecke eigensystems in our purpose. 
Let $\pi_{h_2}$ be the cuspidal automorphic representation associated to $h_2$. Then there exists a lift $\pi_{g}$ of 
$\pi_{h_2}$ to $D^\times(\A_M)$ such that the character of $M^\times\bs \A^\times_M$ corresponding to 
the central character of $\pi_g$ via (\ref{psi}) and a Hecke eigen form $g$ generating $\pi_g$ of weight 
$((2\cdot\textbf{1}_d,3\cdot\textbf{1}_d),w')\in \Z^m\times \Big(\ds\frac{1}{2}\Z\Big)^m$ for some 
$w'\in \Big(\ds\frac{1}{2}\Z\Big)^m$. Notice that for each finite place $v\not\in S\cup\{v|2\}$ 
(see (\ref{fha2} for the set $S$)) the Hecke eigenvalue $a_{T_{2,v}}(g)$ of $T_{2,v}$ for $g|_{H(\A^\infty_M)}$  
belongs to $\O$, but $a_{T_{v}}(g)\not\in \O$ in general and we have $a_{T_{2,v}}(g)\equiv  a^2_{T_{v}}(f)\ {\rm mod}\ m_{\O}$. 
However, it is easy to see that $a_{T_{v}}(g)a_{T_{{}^\iota v}}(g)\in \O$ and its reduction modulo $m_{\O}$ can be recovered 
from the square root of 
$$(a_{T_{v}}(g)a_{T_{{}^\iota v}}(g))^2\equiv a^2_{T_{v}}(f)a^2_{T_{{}^\iota v}}(f)\ {\rm mod}\ m_{\O}$$ which is 
uniquely determined since the characteristic of $\F$ is 2.  
Notice that the weight of $g$ is not parallel weight 2. By Theorem 6.10 of \cite{DLP}, there exists a 
Hilbert modular Hecke eigen cusp form on ${\rm GL}_2(\A_M)$ of weight $((2\cdot\textbf{1}_d,3\cdot\textbf{1}_d),w')$. 
Hence the claim follows from matching with Frobenius eigenvalues and the Brauer-Nesbitt's theorem. 
\end{proof}

\subsection{mod 2 automorphy}\label{auto}
\subsubsection{Symplectic automorphy}
Let $F$ be a totally real field. 
Let $\br:G_F\lra {\rm GSp}_4(\F_2)$ be an irreducible  
mod 2 Galois representation. 
\begin{thm}\label{mt-a5}Suppose that ${\rm Im}(\br)$ is isomorphic to $A_5$ and the degree of $F/\Q$ is even.   
Assume further that for each complex conjugation $c$ of $G_K$, $\br(c)$ is of type $(2,2)$. 
Then there exists a Hilbert-Siegel Hecke eigen cusp form $h$ on ${\rm GSp}_4(\A_F)$ of parallel weight 3 such that $\br\simeq \br_{h,2}$. 
\end{thm}
\begin{proof}
By Proposition \ref{a5Galois}, there exists 
a totally odd irreducible representation $\ot:G_F\lra {\rm SL}_2(\F_4)$ such that 
$\br\simeq {\rm Sym}^3\ot$. By Theorem 2 of \cite{Sasaki}, 
$\ot$ is modular. Further, applying Proposition \ref{cong}, there exists a Hilbert modular cusp form $f$ of 
${\rm GL}_2(\A_F)$ of parallel weight 2 with the trivial central character such that $\br_{f,2}\simeq \ot$.  
Let $\pi$ be the cuspidal automorphic representation of ${\rm GL}_2(\A_F)$ corresponding to $f$. 
Notice that $\pi$ is not of dihedral since ${\rm Im}\br_{f,2}\simeq A_5$. 
By using Kim-Shahidi cubic lift \cite{KS}, there exists a globally generic cuspidal automorphic representation ${\rm Sym}^3\pi$ 
of ${\rm GL}_4(\A_F)$. Since $\pi$ has the trivial central character, we can lift it back to 
a globally generic cuspidal automorphic representation $\Pi'$ of ${\rm GSp}_4(\A_F)$. 
Finally, we apply \cite{GeeT} to switch infinite types of $\Pi'$. 
Then we have a cuspidal automorphic representation $\Pi$ of ${\rm GSp}_4(\A_F)$ such that 
for each infinite place $v$ of $F$, $\Pi_v$ is holomorphic discrete series and its Langlands parameter is given by 
$\phi_{(w;2,1)}$ for some integer $w$ which can be easily deduce from ${\rm Sym}^3(\phi(\pi_v))$ where 
$\phi(\pi_v)$ is the local Langlands parameter of $\pi_v$ given by 
$\phi(\pi_v)(z)=|z|^{-w'}\diag(\Big(\ds\frac{z}{\overline{z}}\Big)^\frac{1}{2},
\Big(\frac{z}{\overline{z}}\Big)^{-\frac{1}{2}})$ for $z\in \C$ and for some integer $w'$. 
\end{proof}

\begin{thm}\label{mt-F20}Suppose that ${\rm Im}(\br)$ is isomorphic to $F_{20}$.  
Assume further that for each complex conjugation $c$ of $G_F$, $\br(c)$ is of type $(2,2)$.  
Then there exists a Hilbert-Siegel cusp form $h$ on ${\rm GSp}_4(\A_F)$ of parallel weight 3 such that $\br\simeq \br_{h,2}$ 
as a representation to ${\rm GL}_4(\bF_2)$. 
\end{thm}
\begin{proof}By assumption, there exists an extension $L\supset M\supset F$ such that 
$M/F$ is a totally quadratic extension, $L/M$ is a totally imaginary quadratic extension, and 
$L/F$ is a Galois extension with the Galois group $C_4$. Then we have 
$$\br\simeq {\rm Ind}^{G_F}_{G_M}\ot,\ \ot={\rm Ind}^{G_M}_{G_L}\overline{\chi}$$
for some character $\overline{\chi}:G_L\lra \bF^\times_{2}$ such that ${\rm Im}\ot\simeq D_{10}$ and 
${\rm Im}\overline{\chi}\simeq C_{5}$. 
By assumption, $\ot$ is totally odd and it comes from a Hilbert modular Hecke eigen cusp form  on 
${\rm GL}_2(\A_M)$ of parallel weight 1. 
Further, applying Proposition \ref{cong}, there exists a Hilbert modular cusp form $f$ of 
${\rm GL}_2(\A_M)$ of parallel weight $(\underline{k},\underline{w}')$ with 
 $$\underline{k}=(2\cdot \textbf{1}_d,4\cdot \textbf{1}_d)\in \Z^{S^{(1)}_M}\times \Z^{S^{(2)}_M}$$ such that 
$f$ has the trivial central character and $\br_{f,2}\simeq \ot$.  
Let $\pi$ be the cuspidal automorphic representation of ${\rm GL}_2(\A_M)$ corresponding to $f$. 
Applying Theorem 8.6 in p.251 of \cite{R}, we have 
a cuspidal automorphic representation $\Pi$ of ${\rm GSp}_4(\A_F)$ such that 
for each infinite place $v$, $\Pi_v$ is holomorphic discrete series. 
It follows from Section 2.1 of 
\cite{Dimi} that for each infinite place $v$ of $M$, the local Langlands parameter is given by 
$$\phi_{\pi_v}(z)=
|z|^{-w'_v}\left\{
\begin{array}{cc}
\diag(\Big(\ds\frac{z}{\overline{z}}\Big)^\frac{1}{2},\Big(\frac{z}{\overline{z}}\Big)^{-\frac{1}{2}}) 
&  (v\in S^{(1)}_M)\\
\diag(\Big(\ds\frac{z}{\overline{z}}\Big)^\frac{3}{2},\Big(\frac{z}{\overline{z}}\Big)^{-\frac{3}{2}}) 
&  (v\in S^{(2)}_M)
\end{array}\right.
$$
It follows this that the local Langlands parameter of $\Pi_v$ is $\phi_{(w;2,1)}$ for some $w\in \Z$. 
Hence we have the claim.  
\end{proof}
\subsubsection{Orthogonal automorphy}Recall that ${\rm GSp}_4(\F_2)={\rm Sp}_4(\F_2)={\rm SO}_4(\F_2)$. 
We discuss mod 2 automorphy of $\br$ when we view it with a representation to ${\rm SO}_4(\F_2)$. 
In the Galois representations associated to regular  algebraic, essentially self-dual, 
cuspidal automorphic representation of ${\rm GL}_n(\A_F)$ for a totally real field $F$, we refer Section 1 of \cite{BGHT}.  
\begin{thm}\label{mt-s5}Suppose that ${\rm Im}(\br)$ is isomorphic to $S_5$.  
Assume further that for each complex conjugation $c$ of $G_F$, $\br(c)$ is of type $(2,2)$. 
Then there exists a regular orthogonal cuspidal automorphic representation $\Pi$ of ${\rm GL}_4(\A_F)$ of weight zero  such that  
$\br\simeq \br_{\Pi,2}$ as a representation to ${\rm GL}_4(\bF_2)$. 
\end{thm}
\begin{proof}
By Proposition \ref{s5Galois}, there exist a cubic character $\overline{\psi}:G_F\lra \F^\times_4$, a totally real quadratic extension $M/F$ 
with the Galois group ${\rm Gal}(M/F)=\langle \iota \rangle$ and 
a totally odd irreducible representation $\ot:G_M\lra {\rm SL}_2(\F_4)$ such that 
$\br\simeq ({\rm Ind}^{G_F}_{G_M}\ot)\otimes\psi$. By Theorem 2 of \cite{Sasaki}, 
$\ot$ is modular. Further, applying Proposition \ref{cong}, there exists a Hilbert modular cusp form $f$ of 
${\rm GL}_2(\A_F)$ of weight $(\underline{k},\underline{w}')$ with 
 $$\underline{k}=(2\cdot \textbf{1}_d,3\cdot \textbf{1}_d)\in \Z^{S^{(1)}_M}\times \Z^{S^{(2)}_M}$$ such that 
the central character of $f$ is a power of $|\det|_{\A^\times_M}$  and $\br_{\pi,2}\simeq 
\ot\otimes {}^\iota\ot$ where $\pi=\pi_f\otimes {}^\iota \pi_f$ is 
the tensor product of $\pi_f$ and its twists ${}^\iota \pi_f$.   
Notice that the central character of $\pi_f$ is ${\rm Gal}(M/F)$-invariant. 
Hence, it can be viewed as a cuspidal automorphic representation of ${\rm GO}(2,2)(\A_F)$. 
Further, by Theorem M, p.54 of \cite{Ra}, it can be transfered to a cuspidal 
automorphic representation $\Pi$ of ${\rm GL}_4(\A_F)$.  
For each infinite place $v$ of $M$, the local Langlands parameter is given by 
$$\phi_{\pi_v}(z)=
|z|^{-w'_v}\left\{
\begin{array}{cc}
\diag(\Big(\ds\frac{z}{\overline{z}}\Big)^\frac{1}{2},\Big(\frac{z}{\overline{z}}\Big)^{-\frac{1}{2}}) 
&  (v\in S^{(1)}_M)\\
\diag(\Big(\ds\frac{z}{\overline{z}}\Big),\Big(\frac{z}{\overline{z}}\Big)^{-1}) 
&  (v\in S^{(2)}_M)
\end{array}\right.
$$
It follows from this that $\Pi$ is of weight zero. Hence for such a $v$, the local Langlands parameter of $\Pi_v$ 
is given by $\phi_{(w;2,1)}:W_\R\lra {\rm GL}_4(\C)$ via ${\rm GSp}_4(\C)\subset {\rm GL}_4(\C)$. 
\end{proof}

\begin{rmk}\label{level}
In the course of the proofs of above theorems, we apply Sasaki's modularity lifting theorem 
which has achieved by raising level. Therefore, we can not specify the levels of Hilbert modular forms 
appearing there. In conclusion, we can not specify the levels of our Hilbert-Siegel modular forms as well. 
However, there is some hope and a substantial expectation for further development of the level lowering method. 
In fact, Conjecture 4.7 and Conjecture 4.9 of \cite{BDJ} show that we can take 
Hilbert modular forms in question to have the correct level regarding Artin conductor of $\br$. 
\end{rmk}

\section{Local mod $2$ computations}
Let us regard $X_\psi$ as a smooth projective scheme over ${\rm Spec}\hspace{0.5mm}\O_K[\frac{1}{5},\psi,\frac{1}{\psi^5-1}]$. 
The mirror variety $W_\psi$ of Dwork quintic $X_\psi$ is defined over 
${\rm Spec}\hspace{0.5mm}\O_K[\frac{1}{5},\psi,\frac{1}{\psi^5-1}]$ by the toric construction in Section \ref{MV} 
and Lemma \ref{triang} in Section \ref{Trc}. 
In this section we devote to calculating the number of rational points on $W_\psi$ modulo $2$ over a finite field.  

Let $p$ be a prime number with $p \ne 2, 5$, $q$ be a positive power of $p$, 
and $\psi \in \mathbb F_q$ such that $\psi^5 \ne 1$. 
The zeta function of $W_\psi$ over $\F_q$ is given by 
$$
    Z_{W_\psi/\F_q}(t) = \frac{P_{\psi, q}(t)}{(1-t)Q_{\psi, q}(t)Q_{\psi, q}(qt)(1-q^3t)}. 
$$
with the numerator 
$$
     P_{\psi, q}(t) = 1 - a_{\psi, q} t + b_{\psi, q} t^2 - q^3a_{\psi, q}t^3 + q^6t^4 \in \mathbb Z[t]
$$
and the polynomial $Q_{\psi, q}(t) = 1 + \cdots \pm q^{101}t^{101} \in \mathbb Z[t]$ of degree $101$ 
such that the absolute values of all reciprocal roots of $Q_{\psi, q}(t) = 0$ are $q$ 
by Poincar\'e duality and the strong Lefschetz formula. 
Recall the polynomial  $f_\psi(x) = 4x^5-5\psi x^4+1$ with the discriminant $2^85^5(1-\psi^5)$. 
\begin{prop}\label{c0} Let $n(f_\psi, q)$ be 
the cardinal of solutions of the equation $f_\psi(x) = 0$ in $\F_q$. Then the congruence below holds:
$$
      a_{\psi, q} = 1 - (Q_{\psi, q}'(0) + qQ_{\psi, q}'(0)) 
      + q^3 - \sharp\, W_\psi(\mathbb F_q)\, \equiv\, n(f_\psi, q) + 1\, (\mathrm{mod}\, 2). 
$$
\end{prop}

Proposition \ref{c0} follows from Lemmas \ref{c3}, \ref{c2} and \ref{c4} below. 

\subsection{Toric Calabi-Yau $3$-folds}\label{MV}
Let us recall a toric construction of $W_\psi$ by Batyrev in \cite{Ba} but we follow the notation 
in Section 6 of \cite{Wa}.  
Let $U_\psi$ be a smooth affine variety defined by the equation below in $\mathbb G_{\mathrm{m}}^4$:
$$
U_\psi:x_1 + x_2 + x_3 + x_4 + \frac{1}{x_1x_2x_3x_4} - 5\psi = 0. 
$$
It follows from $\psi^5-1\neq 0$ that $U_\psi$ is smooth. 
Let $\Delta$ be the convex $4$-dimensional reflexive polyhedron with vertices 
$$
      e_1, e_2, e_3, e_4, -e_1 - e_2 - e_3 - e_4
$$
in $M_{\mathbb R}=\mathbb R^4$., where $e_i$ is the standard basis of $M_{\mathbb R}$. 
Then there exists an $\mathbb R$-linear 
isomorphism $\eta_j : M_{\mathbb R} \rightarrow M_{\mathbb R}\, (j = 1, 2, 3, 4)$ 
which is defined by $\eta_j(e_j) = -e_1 - e_2 - e_3 - e_4$ 
and $\eta_j(e_i) = e_i\, (i \ne j)$. Then $\eta_j(\Delta) = \Delta$. 
Let us define $\mathbb P_\Delta$ 
by the projective toric variety associated to the polyhedron $\Delta$. The projective toric variety $\mathbb P_\Delta$ 
is a singular Fano variety such that 
$$
   \begin{array}{l}
       \mathbb P_\Delta = \cup_{0\leq i \leq 4}U_i\, \, 
       \mbox{where}\, \,  U_i \subset \mathbb A^5\, \mbox{is defined by}\,  y_{i, 1}y_{i, 2}y_{i, 3}y_{i, 4}= y_{i, 0}^5 \\
       \hspace*{15mm} \mbox{\rm with}\, \, \left\{\begin{array}{ll} y_{0, 0} = x_1x_2x_3x_4, y_{0, j} = x_j x_1\cdots x_4 &\mathrm{if}\, i = 0 \\
                                               y_{i, 0}=\frac{1}{x_i}, y_{i, i} = \frac{1}{x_i x_1\cdots x_4}, y_{i, j} = \frac{x_j}{x_i}\, (j \ne 0, i) 
                                               &\mathrm{if}\, i \ne 0. \end{array}\right.
                                      \end{array}
$$
Mote that $\eta_j$'s induce an automorphism $\widetilde{\eta}_j$ of $\mathbb P_\Delta$ such that $\widetilde{\eta}_j(U_k) = U_l$ 
for $(k, l) = (0, j), (j, 0)$ or for $k = l \ne 0, j$. 
$\mathbb P_\Delta$ 
has a stratification $\mathbb P_\Delta = \coprod_{\tau \in \Delta}\mathbb T_{\Delta, \tau}$ where 
$\tau$ runs through all faces of $\Delta$ such that 
$\mathbb T_{\Delta, \tau} \cong \mathbb G_{\mathrm{m}}^{\mathrm{dim}\, \tau}$ since 
$y_{i, 0} = \frac{y_{0, 0}}{y_{0, i}}, y_{i, i} = y_{0, i}, y_{i, j} = \frac{y_{0, j}}{y_{0, i}}\, (j \ne 0, i)$ for $i \ne 0$. 

Let us regard $U_\psi$ as a closed subscheme in $\mathbb T_{\Delta, \Delta} = \mathbb G_\mathrm{m}^4$ with coordinates 
$x_1^{\pm 1}, x_2^{\pm 1}, x_3^{\pm 1}, x_4^{\pm 1}$
and define $Y_\psi$ by the Zariski closure of $U_\psi$ in $\mathbb P_\Delta$. 
The stratification of $\mathbb P_\Delta$ induces 
a stratification 
$$
       Y_\psi = \coprod_{\tau \in \Delta} Y_{\psi, \tau}, \hspace*{3mm} 
       Y_{\psi, \tau} = Y_\psi \cap \mathbb T_{\Delta, \tau} : z_1+\cdots+z_{\mathrm{dim}\tau}+1 = 0\, \, \mbox{in}\, \, 
       T_{\Delta, \tau} \cong \mathbb G_{\mathrm m}^{\mathrm{dim}\tau}
$$
of $Y_\psi$ by smooth subvarieties $Y_{\psi, \tau}$, i.e., $Y_\psi$ is $\Delta$-regular 
\cite[Definition 3.1.1]{Ba}. 
In this situation we obtain a projective smooth mirror symmetry $W_\psi$ of $X_\psi$ 
by an MPCP-desingularization (maximal projective crepant partial desingularization) 
$$
     \pi_\psi : W_\psi \rightarrow Y_\psi
$$
by \cite[Theorem 4.2.2 and Corollary 4.2.3]{Ba} because of $\mathrm{dim}\, W_\psi = 3$. In fact, 
let us take the dual reflexive polyhedron $\Delta^\ast$ which is generated by 
$5e_i^\ast  -e_1^\ast -e_2^\ast-e_3^\ast-e_4^\ast\, (1 \leq i \leq 4)$
$$
      5e_i^\ast  -e_1^\ast -e_2^\ast-e_3^\ast-e_4^\ast\, (1 \leq i \leq 4), -e_1^\ast - e_2^\ast - e_3^\ast - e_4^\ast
$$
in the dual space $N_{\mathbb R}$ of $M_\mathbb R$ \cite[Definition 4.1.1]{Ba}, 
where $\{ e_i^\ast \}$ is the dual basis of $\{e_i\}$. 
We choose a triangulation $\mathcal T^\ast$ of $\Delta^\ast$ such that any codimension $1$ face of $\Delta^\ast$ 
have a triangulation into all basic simplices with vertices in $\Delta^\ast \cap N$. 
Here $N = \langle e_1, \cdots, e_4\rangle \cong \mathbb Z^4$ is the standard lattice of $N_{\mathbb R}$. 
Then 
$\mathcal T^\ast$ is a maximal projective triangulation of $\Delta^\ast$ 
\cite[Definitions 2.2.15, 2.2.18]{Ba} \cite[Chapter III, \S2B]{KKMS}. 
The subdivision $\Sigma(\mathcal T^\ast)$ of the fan $\Sigma(\Delta^\ast)$ associated to 
$\mathcal T^\ast$ and $\Delta^\ast$ \cite[Proposition 2.1.1]{Ba}, respectively, gives an MPCP-desingularization  
$\pi : \mathbb P_{\Sigma(\mathcal T^\ast)} \rightarrow \mathbb P_{\Sigma(\Delta^\ast)}$ and 
$W_\psi$ is the closure of $U_\psi$ in $\mathbb P_{\Sigma(\mathcal T^\ast)}$. 
Moreover the number of rational points over a finite field does not 
depend on the choice of crepant resolutions for projective smooth Calabi-Yau varieties 
\cite[Theorem 2.8]{Ba99}.

\begin{rmk}\label{c6} Batyrev and Kreuzer studied the topological fundamental groups and the Brauer groups for 
toric Calabi-Yau manifolds in \cite{BK}. Fix an embedding $\overline{\mathbb Q} \subset \mathbb C$. 
By their work of integral cohomologies the topological fundamental group $\pi_1^{\mathrm{top}}(W_\psi(\mathbb C))$ 
and the Brauer group $B(W_\psi(\mathbb C))$ of the associated complex manifold $W_\psi(\mathbb C)$ vanish 
by \cite[Corollaries 1.9, 3.9]{BK}. Hence the singular cohomology 
$H^i(W_\psi(\mathbb C), \mathbb Z)$ is torsion-free for any $i$. 
Therefore, the $\ell$-adic etale cohomology $H^i_{\mathrm{et}}(W_{\psi, \overline{\mathbb Q}}, \mathbb Z_\ell)$ 
is torsion-free for any $i$ and any prime number $\ell$ by the comparison theorem between singular and etale torsion cohomologies.
\end{rmk} 

\subsection{Proof of Proposition \ref{c0}}\label{Trc} 

\begin{lem}\label{c3}\mbox{\rm (\cite[\S 6, (11)]{Wa})} 
With the notation as above, we have 
$$
\sharp\, Y_\psi(\mathbb F_q)
= \sharp\, U_\psi(\mathbb F_q) - \frac{(q-1)^4+(-1)^5}{q} + \frac{q^4-1}{q-1}.
$$
In particular, $\sharp\, Y_\psi(\mathbb F_q)\, \equiv\, \sharp\, U_\psi(\mathbb F_q) + 1\, (\mathrm{mod}\, 2)$. 
\end{lem}

\begin{proof}The following is the proof in \cite[\S 6]{Wa} in the case $n = 4$. Since 
$$
\sharp\, Y_{\psi, \tau}(\mathbb F_q) = \sum_{i=1}^{\mathrm{dim}\, \tau}\, 
(-1)^{i-1}\left(\begin{array}{c} \mathrm{dim}\, \tau \\ i-1\end{array}\right)q^\mathrm{\mathrm{dim}\, \tau -i}
= \frac{1}{q}((q-1)^{\mathrm{dim}\, \tau} + (-1)^{\mathrm{dim}\, \tau + 1}) 
$$
by the equation of $Y_{\psi, \tau}$ as above in Section \ref{MV} 
if $1 \leq \mathrm{dim}\, \tau < 4$ and $Y_{\psi, \tau} = \emptyset$ if $\mathrm{dim}\, \tau=0$, we have 
$$
   \begin{array}{lll}
    \sharp\, Y_\psi(\mathbb F_q) &= &\sharp\, U_\psi(\mathbb F_q) - \ds\frac{(q-1)^4+(-1)^5}{q} + \sum_{\tau \in \Delta}\, 
    \ds\frac{1}{q}((q-1)^{\mathrm{dim}\, \tau} + (-1)^{\mathrm{dim}\, \tau + 1}) \\
    &= &\sharp\, U_\psi(\mathbb F_q) - \ds\frac{(q-1)^4+(-1)^5}{q} + \frac{q^4-1}{q-1}.
    \end{array}
$$
\end{proof}

\begin{lem}\label{triang} There exists a triangulation $\mathcal T^\ast$ of $\Delta^\ast$ such that 
all basic simplices consist of vertices in $\Delta^\ast \cap N$ and the cyclic group $C_4$ of order $4$ 
acts on $\mathcal T^\ast$ by $\sigma(e^\ast_i) = e_{\sigma^{-1}(i)}^\ast$ for $i \in \mathbb Z/4\mathbb Z$ and $\sigma \in C_4$. 
\end{lem}

\begin{proof} Let $\Delta^\ast_0$ (resp. $\Delta^\ast_j\, (j = 1, 2, 3, 4)$) be a $3$-dimensional face of $\Delta^\ast$ generated by 
$5e_i^\ast  -e_1^\ast -e_2^\ast-e_3^\ast-e_4^\ast\, (1 \leq i \leq 4)$ (resp. 
$-e_1^\ast -e_2^\ast-e_3^\ast-e_4^\ast, 5e_i^\ast  -e_1^\ast -e_2^\ast-e_3^\ast-e_4^\ast\, (i \in \{ 1, 2, 3,  4\} \setminus \{ j \}$). 
It is sufficient to construct a triangulation $\mathcal T^\ast_0$ of the face 
$\Delta^\ast_0$ of $\Delta^\ast$ sych that $\mathcal T^\ast_0$ consists of basic simplices with vertexes in $\Delta^\ast_0 \cap N$ 
and the cyclic group $C_4$ of order $4$ acts on $\mathcal T^\ast_0$ as in the statement. 
Indeed, if $\eta^\ast_j : N_{\mathbb R} \rightarrow N_{\mathbb R}$ is a dual linear map of $\eta_j$, 
then $\mathcal T^\ast_j = \eta_j^\ast(\mathcal T^\ast_0)$ is 
a triangulation of the face 
$\Delta^\ast_j$ of $\Delta^\ast$ which consists of basic simplices with vertexes in $\Delta^\ast_j \cap N$. 
Moreover, since $\eta_j^\ast|_{\Delta_0^\ast\cap\Delta_j^\ast} = \mathrm{id}_{\Delta_0^\ast\cap\Delta_j^\ast}$ 
and $(\eta_j^\ast)^2 = \mathrm{id}_{N_{\mathbb R}}$, $\mathcal T^\ast = \cup_{0 \leq j \leq 4}\mathcal T^\ast_j$ 
forms a triangulation $\mathcal T^\ast$ of $\Delta^\ast$ such that 
all basic simplices consist of vertices in $\Delta^\ast \cap N$. Moreover, the induced liner map 
$\sigma^\ast : N_{\mathbb R} \rightarrow N_{\mathbb R}$ from $\sigma \in (1, 2, 3, 4) \in C_4$ is 
given by  
$\sigma^\ast = \eta_1^\ast\circ\eta_4^\ast\circ\eta_3^\ast\circ\eta_2^\ast\circ\eta_1^\ast$
so that the cyclic group $C_4$ acts on $\mathcal T^\ast$ by the restriction. 

Let us explain how to construct a triangulation $\mathcal T^\ast_0$ of $\Delta_0^\ast$. The volume of the $3$-dimensional face 
$\Delta_0^\ast$ is $\frac{125\sqrt{2}}{3}$. Note that the volume of the $3$-dimensional 
simplex generated by $e_1^\ast, \cdots, e_4^\ast$ is $\frac{\sqrt{2}}{3}$. 
Let $y_1, \cdots, y_4$ be a system of standard coordinates of $N_{\mathbb R}$. 
Consider the $16$ partitions of $N_{\mathbb R}$ by $4$ hyperplanes defined by $y_1 =0, y_2 = 0, y_3 = 0$ and $y_4 = 0$ including 
boundaries, 
and denotes each of them by $(\pm, \pm, \pm, \pm)$. 

(I) Case $(+, +, +, +)$ : In this case there exists a unique $3$-dimensional basic simplex generated by 
$(1, 0, 0, 0), (0, 1, 0, 0), (0, 0, 1, 0), (0, 0, 0, 1)$ on $\Delta^\ast_0 \cap (+, +, +, +)$. 

(II) Case $(-, +, +, +)$ : The volume of $\Delta^\ast_0 \cap (-, +, +, +)$ is $\frac{7\sqrt{2}}{3}$. 
Consider the subdivision of the partition 
by $\alpha_j \leq y_j \leq \alpha_j + 1\, (j = 2, 3, 4)$ for $\alpha_j =0, 1$. 
\begin{list}{}{}
\item[(II-i)] If $\alpha_2=\alpha_3=\alpha_4 = 0$, then $6$ points 
$$
(-1, 1, 1, 0), (-1, 1, 0, 1), (-1, 0, 1, 1), (0, 1, 0, 0), (0, 0, 1, 0), (0, 0, 0, 1)
$$
belong $\Delta_0^\ast \cap N$. 
Let us connects $2$ points of $6$ points by an edge if either (a) $y_i=1$ for some $2 \leq i \leq 4$, (b) $y_1=0$ or (c) 
two points are $(-1, 1, 0, 1)$ and $(0, 0, 1, 0)$. Here the edge between $(-1, 1, 0, 1)$ and $(0, 0, 1, 0)$ is drawn by our choice. 
Then there exists $4$ pieces of $3$-dimensional basic simplices. 
\item[(II-ii)] If $\alpha_2= 1$ and $\alpha_3=\alpha_4 = 0$, then 
 there exists a unique $3$-dimensional basic simplex generated by 
$(-1, 2, 0, 0), (-1, 1, 1, 0), (-1, 1, 0, 1), (0, 1, 0, 0)$ on $\Delta^\ast_0$. 
\item[(II-iii)] Otherwise, the intersection with $\Delta^\ast_0$ is of dimension less than $3$ and 
they are union of faces of $3$-dimensional simplices obtained in (II-i), (II-ii) and (II-iii). 
\end{list}
Hence we obtain $4 + 1 \times 3 = 7$ basic simplices. By the action of $C_4$ 
we have $28$ basic simplices on $\Delta^\ast_0$ in $(-, +, +, +) \cup (+, -, +, +) \cup (+, +, -, +) \cup (+, +, +, -)$. 

(III) Case $(-, -, +, +)$ : The volume of $\Delta^\ast_0 \cap (-, -, -, +)$ is $4\sqrt{2}$ 
such that $\Delta^\ast_0 \cap N \cap (-, -, +, +)$ consists of 
$12$ points. 
\begin{list}{}{}
\item[(III-i)] If $0 \leq y_3, y_4 \leq 1$, then there exists a unique $3$-dimensional basic simplex generated by 
$(-1, 0, 1, 1), (0, -1, 1, 1), (0, 0, 1, 0), (0, 0, 0, 1)$ on $\Delta^\ast_0$. 

\item[(III-ii)] If $0 \leq y_3 \leq 1$ and $1 \leq y_4 \leq 2$, then $6$ points 
$$
(-1, -1, 1, 2), (-1, 0, 0, 2), (0, -1, 0, 2), (-1, 0, 1, 1), (0, -1, 1, 1), (0, 0, 0, 1)
$$
belong $\Delta_0^\ast \cap N$. Connects two of points for either $y_1, y_2, y_3$ or $y_4$ coincides with each other by edges. 
We also connect points $(-1, -1, 1, 2)$ and $(0, 0, 0, 1)$ by an edge. Then the area is divided by $4$ pieces of 
 $3$-dimensional basic simplices. Note that the edge between $(-1, -1, 1, 2)$ and $(0, 0, 0, 1)$ is our choice. 
We also have a similar division if $1 \leq y_3 \leq 2$ and $0 \leq y_4 \leq 1$ by permutation $y_3$ and $y_4$. 

\item[(III-iii)] If $0 \leq y_3 \leq 1$ and $2 \leq y_4 \leq 3$, 
then there exists a unique $3$-dimensional basic simplex generated by 
$(-1, -1, 0, 3), (-1, -1, 1, 2), (-1, 0, 0, 2), (0, -1, 0, 2)$ on $\Delta^\ast_0$. 

\item[(III-iv)] If $1 \leq y_3, y_4 \leq 2$, then there exists a unique $3$-dimensional basic simplex generated by 
$(-1, -1, 1, 2), (-1, -1, 2, 1), (-1, 0, 1, 1), (0, -1, 1, 1)$
on $\Delta^\ast_0$. 
\item[(II-iii)] Otherwise, the intersection with $\Delta^\ast_0$ is of dimension less than $3$ and 
they are union of faces of $3$-dimensional simplices obtained in (III-i), (III-ii), (III-iii), and (III-iv). 
\end{list}
Hence we obtain $1 + 1 \times 2 + 4 \times 2 + 1 \times 2 + 1 = 12$ basic simplices. By the action of $S_4$ 
we have $24$ basic simplices on $\Delta^\ast_0$ 
in $(-, -, +, +) \cup (-, +, -, +) \cup (-, +, +, -) \cup (+, -, -, +) \cup (+, -, +, -)\cup (+, +, -, -)$.

(IV) Case $(-, -, -, +)$ : The volume of $\Delta^\ast_0 \cap (-, -, -, +)$ is $2$ 
such that $\Delta^\ast_0 \cap N \cap (-, -, -, +)$ consists of 
$8$ points 
$$
(-1, -1, -1, 4), (-1, -1, 0, 3), (-1, 0, -1, 3), (0, -1, -1, 3), (-1, 0, 0, 2), (0, -1, 0, 2),  (0, 0, -1, 2), (0, 0, 0, 1). 
$$
\begin{list}{}{}
\item[(IV-i)] If $1 \leq y_4 \leq 2$, then there exists a unique $3$-dimensional basic simplex generated by 
$(-1, 0, 0, 2), (0, -1, 0, 2), (0, 0, -1, 2), (0, 0, 0, 1)$ on $\Delta^\ast_0$. 

\item[(IV-ii)] If $2 \leq y_4 \leq 3$, then the situation is similar to the case (II-i). 
Here we connect two points $(-1, 0, -1, 3), (0, -1, 0, 2)$ as our choice. Then there exists 
$4$ pieces of $3$-dimensional basic simplices. 
\item[(IV-iii)] If $3 \leq y_4 \leq 4$, then there exists a unique $3$-dimensional basic simplex generated by 
$(-1, -1, -1, 4), (-1, -1, 0, 3), (-1, 0, -1, 3), (0, -1, -1, 3)$ on $\Delta^\ast_0$. 
\item[(IV-iv)] Otherwise, the intersection with $\Delta^\ast_0$ is of dimension less than $3$ and 
they are union of faces of $3$-dimensional simplices obtained in (IV-i), (IV-ii) and (IV-iii). 
\end{list}
Hence we obtain $1 + 4 + 1 =6$ basic simplices. By the action of $C_4$ 
we have $24$ basic simplices on $\Delta^\ast_0$ in $(-, -, -, +) \cup (-, -, +, -) \cup (-, +, -, -) \cup (+, -, -, -)$. 

(V) Case $(-, -, -, -)$ : This case never happens. 

Now we have obtained 
$$
      1 + 28 + 72 + 24 = 125
$$
basic simplices on $\Delta^\ast_0$. By our construction the $3$-dimensional basic simplices above 
and induced faces from them form a desired triangulation $\mathcal T_0\ast$ 
of $\Delta_0^\ast$ on which the cyclic group $C_4$ naturally acts. This completes a proof. 
\end{proof}

\begin{lem}\label{c2} With the notation as above, the congruence 
$\sharp\, W_\psi(\mathbb F_q)\, \equiv\, \sharp\, Y_\psi(\mathbb F_q)\, (\mathrm{mod}\, 2)$ holds. 
\end{lem}

\begin{proof} Let us now fix a triangulation $\mathcal T^\ast$ of $\Delta^\ast$ as in Lemma \ref{triang}. 
Note that the number of $\mathbb F_q$-rational points of 
$W_\psi$ is independent of the choices of triangulations (see the paragraph before Remark \ref{c6}). 
If $Y_\psi^{\mathrm{reg}}$ (resp. $Y_\psi^{\mathrm{sing}}$) denotes the regular (resp. singular) locus of 
$Y_\psi$, then $\pi_\psi$ induces an isomorphism $\pi_\psi^{-1}(Y_\psi^{\mathrm{reg}}) \cong Y_\psi^{\mathrm{reg}}$ and 
$Y_\psi^{\mathrm{reg}} = \coprod_{\mathrm{dim}\, \tau \geq 3} Y_{\psi, \tau}$. 
Then the cyclic group $C_4$ of order $4$ acts on the toric structure of $Y_\psi$ by permutation of coordinates 
and it acts compatibly on the triangulation $\mathcal T^\ast$ of $\Delta^\ast$. 
Let us observe the cardinal of the $C_4$-orbits of rational points of $W_\psi$. 
Each single orbit is lying over the inverse image $\pi^{-1}_\psi(Y_\psi^{\mathrm{reg}})$ of the smooth locus. Therefore, 
on $\pi^{-1}_\psi(Y_\psi^{\mathrm{sing}})$,  
the cardinal of each $C_4$-orbit is either $2$ or $4$ so that 
$$
     \sharp\, W_\psi(\mathbb F_q)\, \equiv\, \sharp\, Y_\psi^{\mathrm{reg}}(\mathbb F_q) + 
     \sharp\, \pi_\psi^{-1}(Y_\psi^{\mathrm{sing}})(\mathbb F_q) \equiv\, \sharp\, Y_\psi^{\mathrm{reg}}(\mathbb F_q) 
     \equiv\,  \sharp\, Y_\psi(\mathbb F_q) - \sharp\, Y_\psi^{\mathrm{sing}}(\mathbb F_q)\, (\mathrm{mod}\, 2).
$$
The proof completes after the congruence 
$$
     \sharp\, Y_\psi^{\mathrm{sing}}(\mathbb F_q) = \sum_{\mathrm{dim}\, \tau \leq 2}\sharp\, Y_{\psi, \tau}(\mathbb F_q) 
     = 10(q-1) + 10 = 10q\, \equiv\, 0\, (\mathrm{mod}\, 2).  
$$
\end{proof}

\begin{lem}\label{c4} 
$\sharp\, U_\psi(\mathbb F_q)\, \equiv\, n(f_\psi, q)\, (\mathrm{mod}\, 2).$
\end{lem}

\begin{proof} By the natural $C_4$-action on $U_\psi$ inducing by permutation of coordinates. 
\end{proof}

\begin{rmk} One can prove the congruence $a_{\psi, q}\, \equiv\, n(f_\psi, q) + 1\, (\mathrm{mod}\, 2)$ 
in Proposition \ref{c0} using the formula of zeta function of $X_\psi$ 
for $\psi^5 \ne 0, 1$, which was conjectured in \cite{COR} and proved in \cite{Go}: 
$$
     Z_{X_\psi/\mathbb F_q}(t) = \frac{P_{\psi, q}(t)P_{A_\psi/\mathbb F_q}(qt)^{10}
     P_{B_\psi/\mathbb F_q}(qt)^{15}}
     {(1-t)(1-qt)(1-q^2t)(1-q^3t)}
$$
where $A_\psi$ and $B_\psi$ are projective smooth curves defined by the smooth completions of the equations 
$$
     y^5 = x^2(1-x)^3(x-\psi^5)^2\, \,  \mbox{\rm and}\, \, y^5 = x^2(1-x)^4(x-\psi^5), 
$$
respectively. Indeed, one has $\sharp\, X_\psi(\mathbb F_q)\, \equiv\, n(f_\psi, q) + 1\, (\mathrm{mod}\, 2)$ by a direct computation, and 
$P_{B_\psi/\mathbb F_q}(t)$ is a square of polynomial of degree $4$ if $5\, |\, q-1$ since $B_\psi$ is furnished with an involution 
defined by $(x, y) \mapsto (\psi^5/x, \psi^3(1-x)(x-\psi^5)/(xy))$ and $P_{B_\psi/\mathbb F_q}(t)$ 
is a product of $L$-functions of character sums
$$
      x \mapsto \chi^i_r\left(x^2(1-x)^4(x-\psi^5)\right)
$$
on $\mathbb P^1 \setminus \{0, 1, \infty, \psi^5\}$ 
for $1 \leq i \leq 4$ where $\chi : \mathbb F_q^\times \rightarrow \mathbb C^\times$ is a primitive multiplicative character of order $5$ 
and $\chi^i_r = \chi^i\circ N_{\mathbb F_{q^r}/\mathbb F_q}$. The similar also holds for $\psi = 0$. 
\end{rmk} 

\section{Global monodromy}

Let $K$ be a number field, and $\psi \in K$ such that $\psi^5 \ne 1$. 
We denote by $L_\psi$ the fixed field of the kernel $\mathrm{Ker}(\overline{\rho}_{\psi, 2})$ 
of the mod $2$ representation 
$$
    \overline{\rho}_{\psi, 2} : G_K \rightarrow 
      {\rm GSp}(\overline{T}_{\psi,2},\langle \ast,\ast \rangle_{\F_2}), 
$$
where $T_{\psi,2}$ is a lattice of $V_{\psi, 2} = H_{\mathrm{et}}^3(W_{\psi, \overline{\mathbb Q}}, \mathbb Q_2)$ and 
$\overline{T}_{\psi,2} = T_{\psi,2} \otimes_{\Z_2}\F_2$. Let us fix an isomorphism 
${\rm GSp}(\overline{T}_{\psi,2},\langle \ast,\ast \rangle_{\F_2}) \cong S_6$. 
In this section we prove Theorem \ref{image} (1), that is, to determine the global monodromy group $\mathrm{Im}(\overline{\rho}_{\psi, 2})$. 

\subsection{Euler factors and global monodromy modulo $2$} At first we fix notation. 
Let $v$ be a finite place with $v\nmid 10$ 
such that $W_\psi$ has a good reduction at $v$, 
that is, $\psi^5 - 1$ is a $v$-adic unit. Let us define the local Euler factor 
$$
    P_{\psi, v}(t) = 1 - a_{\psi, v}t + b_{\psi, v}t^2 - q_v^3a_{\psi, v}t^3 + q_v^6t^4 \in \mathbb Z[t]
$$ 
of $V_{\psi, 2}$ at $v$, where $\F_v$ is the residue field of $K$ at $v$ and $q_v$ 
is the cardinal of $\F_v$. 
We denote the mod $2$ Galois representation 
$$
\overline{\rho}_{\psi, v, 2} : \mathrm{Gal}(\overline{\mathbb F}_v/\mathbb F_v) 
\rightarrow {\rm GSp}(\overline{T}_{\psi,2},\langle \ast,\ast \rangle_{\F_2}). 
$$
for a finite place $v$ with good reduction via $G_K \supset D_{K, v} \rightarrow \mathrm{Gal}(\overline{\mathbb F}_v/\mathbb F_v)$. 
Here $D_{K, v}$ is the decomposition group at the place $v$. We also denote the geometric Frobenius at $v$ by $\mathrm{Frob}_v$ and then 
$$
 P_{\psi, v}(t) = \mathrm{det}(1- t\rho_{\psi, \ell}(\mathrm{Frob}_v); V_{\psi, \ell}) 
$$
for a prime number $\ell$ with $v\, \not|\, \ell$. 

Our initial idea to determine the monodromy group ${\rm Im}(\br_{\psi, 2})$ came from an observation 
of the numerical table in \cite{COR}, and we had a following question. 

\begin{qn}\label{b0} \mbox{\rm (Theorem \ref{s11})}
If $a_{\psi, v}$ is even, then so is $b_{\psi, v}$. In particular, the possible  characteristic polynomial $P_{\psi, v}(t)$ modulo $2$ 
is one of the following:
$$
     1 + t^4, \, \, 1+t + t^3 + t^4, \, \, 1+ t + t^2 +t^3 +t^4. 
$$
\end{qn}

An immediate observation shows the following:
\begin{prop}\label{b1} Assume that the question above holds for $\psi$. 
Any conjugacy of $(123)(456)$ 
in $S_6 \cong {\rm GSp}(\overline{T}_{\psi,2},\langle \ast,\ast \rangle_{\F_2})$ 
does not appear in the image of $\overline{\rho}_{2, v}$. 
In particular, the image $\mathrm{Im}(\overline{\rho}_{\psi, 2})$ 
 is isomorphic to a subgroup of $S_5$ or $S_4\times S_2$. 
\end{prop}

\begin{proof}
Let $L_\psi$ be the fixed field of the kernel $\mathrm{Ker}(\overline{\rho}_{\psi, 2})$ in $\overline {\mathbb Q}$.
Let us regard $\mathrm{Im}(\overline{\rho}_{\psi, 2})$ as a subgroup
of $S_6$. Then the characteristic polynomial of the conjugacy of $\sigma = (123)(456) \in S_6$ is $1 + t^2 + t^4$.
Suppose that $\mathrm{Im}(\overline{\rho}_{\psi, 2})$ contains an element conjugated to $\sigma = (123)(456)$.
Then there exist infinitely many finite places $v$ of $K$ which are unramified such that, for such a $v$  
the characteristic polynomial $P_{\psi, v}(t)$ of the geometric Frobenius $\mathrm{Frob}_v$ at $v$ is $1 + t^2 + t^4$
modulo $2$ by Chebotarev's density theorem.
It contradicts the assertion of the question.
Hence, any conjugacy of $\sigma=(123)(456)$ is not contained in $\mathrm{Im}(\overline{\rho}_{\psi, 2})$.
Moreover, it means the order of $\mathrm{Im}(\overline{\rho}_{\psi, 2})$
divides $240$. Then the assertion follows from the classification of subgroups of $S_6$.
\end{proof}

Henceforth, we will try to relate reciprocity of the quintic trinomial $f_\psi$ with 
the global monodromy $\mathrm{Im}(\overline{\rho}_{\psi, 2})$. As a byproduct, we will give an affirmative answer to the above question in Theorem \ref{s11}.  

\subsection{Local to global} Let $f_\psi(x) = 4x^5-5\psi x^4+1$ be a polynomial over $K$, 
and $K_{f_\psi}$ be the decomposition field of $f_\psi$ over $K$. 

\begin{prop}\label{s1}
The degree of extension of the composite $K_{f_\psi}L_\psi$ over $L_\psi$ is a power of $2$. 
\end{prop}

\begin{proof}Let $M$ be a subextension of $K_{f_\psi}L_\psi/L_\psi$, and $v$ a finite place of $M$ 
such that $\psi^5 - 1$ is a unit at $v$. Then the characteristic polynomial $P_{\psi, v}(t)$ 
of the geometric Frobenius $\mathrm{Frob}_v$ is $1+t^4$ modulo $2$ 
since $\overline{\rho}_{\psi, v, 2}(\mathrm{Frob}_v)$ acts trivially 
by $L_\psi \subset M$. Hence the equation $f_\psi = 0$ admits an odd number of solutions at each finite place of $M$ 
as above by Proposition \ref{c0}. 
Hence the equation $f_\psi=0$ has an odd number of solutions in $M$ by Chebotarev's density theorem. 
Since $K_{f_\psi}L_\psi = (L_\psi)_{f_\psi}$, $\mathrm{deg}(K_{f_\psi}L_\psi/L_\psi)$ is a power of $2$. 
\end{proof}

\begin{prop}\label{s7} The degree of extension of the composite $K_{f_\psi}L_\psi$ over $K_{f_\psi}$ 
is a multiple of a power of $2$ 
and a power of $3$. 
\end{prop}

\begin{proof} For any place $v$ of $K_{f_\psi}$ 
such that $\psi^5 - 1$ is a unit at $v$, the characteristic polynomial $P_{\psi, v}(t)$ 
of the geometric Frobenius $\mathrm{Frob}_v$ is either $1 + t^2+ t^4$ or $1+t^4$ modulo $2$ 
by Proposition \ref{c0} since $n(f_\psi, q_v) = 5$. Hence its order is $2^i \times 3^j$ for some $i, j \geq 0$. 
Since the order of extension of the residue fields at almost all places in the extension 
$K_{f_\psi}L_\psi/K_{f_\psi}$ is  $2^i \times 3^j$ for some $i, j \geq 0$, 
the degree of the extension $K_{f_\psi}L_\psi/K_{f_\psi}$ is a multiple of a power of $2$ 
and a power of $3$ by Chebotarev's density theorem. 
\end{proof}
For any finite extension $L/M$ of fields, let us write its degree for 
$\mathrm{deg}(L/M)=[L:M]$.  
\begin{cor} It holds that 
\begin{enumerate}
\item $\mathrm{deg}(L_\psi/K)$ is divided by $5$ if and only if 
so is $\mathrm{deg}(K_{f_\psi}/K)$; 
\item If $\mathrm{deg}(K_{f_\psi}/K)$ is divided by $3$, then so is $\mathrm{deg}(L_\psi/K)$. 
\end{enumerate}
\end{cor}

\begin{proof}The assertion follows from Propositions \ref{s1} and \ref{s7} and the equality 
$$
\mathrm{deg}(K_{f_\psi}L_\psi/K_{f_\psi}) \mathrm{deg}(K_{f_\psi}/K) 
 = \mathrm{deg}(K_{f_\psi}L_\psi/K) = \mathrm{deg}(K_{f_\psi}L_\psi/L_\psi)\mathrm{deg}(L_\psi/K). 
$$ 
\end{proof}

\subsection{Non-fullness}

\begin{thm}\label{s3} Under the natural inclusion, we have $\mathrm{Im}(\overline{\rho}_{\psi, 2}) \ne S_6, A_6$. 
\end{thm}

\begin{proof}Suppose $\mathrm{Im}(\overline{\rho}_{\psi, 2})$ is isomorphic to either $S_6$ or $A_6$. 
Consider the natural morphism 
$$
       \nu : {\rm Gal}(K_{f_\psi}L_\psi/K_{f_\psi}) \rightarrow {\rm Gal}(L_\psi/K)
$$
by restriction. Then $\nu$ is injective and $\mathrm{Im}(\nu)$ is a normal subgroup 
of ${\rm Gal}(L_\psi/K)$ by Galois theory. Since the order of $\mathrm{Im}(\nu)$ is a multiple of a power of $2$ 
and a power of $3$ by Proposition \ref{s7}, $\mathrm{Im}(\nu)$ is trivial by Lemma \ref{s6} below. 
Hence $L_\psi \subset K_{f_\psi}$. Since ${\rm Gal}(K_{f_\psi}/K)$ is a subgroup of $S_5$, 
one has $\mathrm{Im}(\overline{\rho}_{\psi, 2}) \not\cong\, S_6, A_6$. 
\end{proof}

\begin{lem}\label{s6} Let $H$ be a subgroup of $S_6$ such that 
the order of $H$ is divided by $5$. Then any normal subgroup of $H$ whose order is not divided by $5$ is trivial.
\end{lem}

\begin{proof}It follows from  Section \ref{subS6} that $H$ is isomorphic to 
one of $C_5, D_{10} = C_4 \ltimes C_5, F_{20} = C_4 \ltimes C_5, A_5, S_5, A_6$ and $S_6$. 
Then the conclusion is a well-known fact.
\end{proof} 

\subsection{Irreducible case}

We determine the mod $2$ monodromy group when $f_\psi$ is irreducible. 

\begin{thm}\label{s5} If $f_\psi$ is irreducible over $K$, then $L_\psi = K_{f_\psi}$, that is, 
$$
     \mathrm{Im}(\overline{\rho}_{\psi, 2}) \cong \mathrm{Gal}(K_{f_\psi}/K). 
$$
\end{thm}

\begin{lem}\label{s2} Suppose that $f_\psi$ is irreducible over $K$. Then $K_{f_\psi}$ is a subfield of $L_\psi$. 
\end{lem}

\begin{proof}Since $f_\psi$ is irreducible, 
$\mathrm{Gal}(K_{f_\psi}/K)$ is a subgroup of $S_5$ whose order is divided by $5$. 
Consider the surjection
$$
    \theta :  \mathrm{Gal}(K_{f_\psi}L_\psi/K) \rightarrow \mathrm{Gal}(K_{f_\psi}/K). 
$$
Then the image $\theta(\mathrm{Gal}(K_{f_\psi}L_\psi/L_\psi))$ is a normal subgroup of $\mathrm{Gal}(K_{f_\psi}/K)$ 
of order a power of $2$ by Proposition \ref{s1}. $\theta(\mathrm{Gal}(K_{f_\psi}L_\psi/L_\psi))$ 
is trivial by Lemma \ref{s6} so that $\mathrm{Gal}(K_{f_\psi}L_\psi/L_\psi)$ is trivial. 
It concludes $K_{f_\psi} \subset L_\psi$. 
\hspace*{\fill} 
\end{proof}

\vspace*{1mm}

\noindent
{\sc Proof of Theorem \ref{s5}.} 
Consider the canonical surjection $\theta : \mathrm{Gal}(L_\psi/K) \rightarrow \mathrm{Gal}(K_{f_\psi}/K)$ by Lemma \ref{s2}. 
The order of the kernel $\mathbf{\theta}$ is a multiple of a power of $2$ and a power of $3$ by Proposition \ref{s7}. 
Since our hypothesis of the irreducibility of $f_\psi$ 
implies that $\mathrm{deg}(K_{f_\psi}/K)$ is divided by 
$5$, the kernel $\mathbf{\theta}$ must be trivial by Lemma \ref{s6}.
\hspace*{\fill} $\Box$







\subsection{Disappearance of $1+t^2 +t^4$}\label{parity} We give an answer to Question \ref{b0}.

\begin{thm}\label{s11} If $a_{\psi, q}$ is even, then so is $b_{\psi, q}$. 
\end{thm}

\begin{proof} First we prove the assertion when $f_\psi$ is irreducible. 
Suppose $a_{\psi, v}$ is even and $b_{\psi, v}$ is odd. Then the order of the image 
$\overline{\rho}_{\psi, 2}(\mathrm{Frob}_v)$ of geometric Frobenius is either $3$ or $6$. 
Since $\mathrm{Gal}(K_{f_\psi}/K) \cong \mathrm{Im}(\overline{\rho}_{\psi, 2})$ by Theorem \ref{s5}, 
the type of polynomial $f_\psi$ has a factor of degree $3$ over $\mathbb F_v$. 
It contradicts the congruence $a_{\psi, v}\, \equiv\, n(f_\psi, v)+ 1\, (\mathrm{mod}\, 2)$ in Proposition \ref{c0} 
since $n(f_\psi, q_v) = 0$ or $2$. 

In general there exists an element $\psi' \in K$ such that $\psi'\, \equiv \psi\, (\mathrm{mod}\, v)$ and $f_{\psi'}$ 
is irreducible over $K$ by the weak approximation theorem and Eisenstein's criterion of irreducibility 
for a place $w$ of $K$ above $5$. Since $P_{\psi, v}(t) = P_{\psi'. v}(t)$, the assertion follows from that of the irreducible case. 
\end{proof}

\section{Rational points on some Diophantine equations}\label{RT}
In this section, we will prove the following with an elementary method without using 
Coleman-Chabauty:
\begin{thm}\label{completed-version}
For each $\psi\in \Q\setminus\{0,1\}$ such that $f_\psi(x)=4x^5-5\psi x^4+1$ is irreducible over $\Q$, 
it holds that ${\rm Gal}(\Q_{f_\psi}/\Q)\simeq S_5$.  
\end{thm}
It follows from 
Theorem \ref{no-rat-F20} and Theorem \ref{no-rat-D10}. To exclude the possible types of $\mathrm{Im}(\overline{\rho}_{\psi, 2})$ other than $S_5$ we relate it with 
rational points on some Diophantine equations. 

\subsection{On $D_{10}$ and $F_{20}$}
Let $\br=\br_{\psi,2}:G_\Q\lra {\rm GSp}_4(\F_2)$ be the mod 2 Galois representation 
associated to $W_{\psi}$ with $\psi\in \Q\setminus \{0,1\}$. 
Assume that $f_\psi$ is irreducible over $\Q$ and this condition yields 
${\rm Gal}(\Q_{f_\psi}/\Q)\not\simeq C_5$, since $f_\psi$ has complex roots which are not real. 
By Theorem \ref{image}, ${\rm Im}(\br)\simeq {\rm Gal}(\Q_{f_\psi}/\Q)\subset S_5$.  Assume further that
for the complex conjugation $c$ of $G_\Q$, $\br(c)$ is of type $(2,2)$ as an element of $S_5$. 
This is equivalent to $\psi<1$.  
As is proved later that the discriminant of $f_\psi$ is non-square in $\Q^\times$. 
In particular, this excludes the cases of $D_{10}$ and $A_5$ among possible types of ${\rm Im}(\br)$. 

Let us first observe the cases of $D_{10}$ and $F_{20}$. 
Put $g_\psi(x):=x^5f_\psi(x^{-1})=x^5-5\psi x+4$. 
Suppose that $g_\psi$ is of type $D_{10}$. As explained at the second paragraph in p.138 of \cite{RYZ}, due to Weber and Chebotarev,   
there exist two rational functions $f_1(\lambda,\mu),\ f_2(\lambda,\mu)$ in two variables $\lambda,\mu$ with 
$\lambda\neq 1$ and $\mu\neq 0$ 
such that 
$$f_1(\lambda,\mu):=\frac{5^5\lambda\mu^4}{(\lambda-1)^4(\lambda^2-6\lambda+25)}=-5\psi,\ f_2(\lambda,\mu):=\frac{5^5\lambda\mu^5}{(\lambda-1)^4(\lambda^2-6\lambda+25)}=4.$$
Regarding the second equation $f_2(\lambda,\mu)=4$, let us substitute the variables suitably we have the 
following affine equation 
$$Y^2=X^{10}+11X^5-1,\ 
(X,Y)=\Big(\frac{-8 (-3+\lambda ) (-1+\lambda )^5+5^5 (-1+2 \lambda ) \mu ^5}{2^5(\lambda-1)^5}, 
\frac{5 \mu }{2 (-1+\lambda )}\Big).$$

Next we consider the case of $F_{20}$. 
Assume that ${\rm Gal}(\Q_{f_\psi}/\Q)\simeq F_{20}$. 
Applying Theorem 1 of \cite{Du} to $g_{\psi}$, we have the equation 
$$0=x^6-40 x^5 \psi +1000 x^4 \psi ^2-20000 x^3 \psi ^3+250000 x^2 \psi ^4+4000000 \psi  (3+\psi ^5)-800000 x 
(1+2 \psi ^5).$$
Notice that $x=10\psi$ implies $\psi=0$ in which case we have ${\rm Gal}(\Q_{f_\psi}/\Q)\simeq F_{20}$. 
Therefore, we may assume that $\psi\neq 0$ and hence $x\neq 10\psi$. 
Substituting variables we have 
$$Y^2=X^{10}+11X^5-1,\ 
X=\frac{10}{x-10\psi},\ Y=\frac{x \{800000-(x-10 \psi )^5\}-10000000 \psi }{20 (x-10 \psi )^5 \psi }$$
again. 
 
Let $C$ be a unique smooth completion of the affine hyperelliptic curve $W:Y^2=X^{10}+11X^5-1$ and 
$J$ be its Jacobian. 
Clearly, $C(\Q)=W(\Q)\cup\{\infty_{\pm }\}$. For each positive integer $n$, let $J[n]$ be 
the group scheme consisting of $n$-torsion points.  
\begin{thm}\label{no-rat-F20}It holds that $C(\Q)=\{\infty_{\pm }\}$. In particular, $W(\Q)=\emptyset$. 
\end{thm}
\begin{proof}Put $F=\Q(\zeta_5)$. Notice that $11$ is split completely in $F$ and for 
each finite place $w$ lying over $11$, $D$ has good reduction at $w$. 
Pick such a place $w$. Notice that $\F_w\simeq \F_{11}$ and 
we fix this isomorphism. 
Since $C$ is of genus greater than 1, we have an injection 
$$r_{\infty_+}:C(\bQ)\hookrightarrow J(\bQ),\ Q\mapsto Q-\infty_+$$
which is defined over $\Q$. 
By using Magma \cite{BCP}, we see that $J(\Q)$ is of Mordell-Weil rank 0. Hence $J(\Q)=J(\Q)_{{\rm tor}}$ where 
the subscript ``tor" means the maximal torsion subgroup of $J(\Q)$. 
Put $T=:\{(x,0)\in W(\bQ)\}\cup\{\infty_{\pm}\}$. Then $T \subset C(F)$ since 
$$X^{10}+11X^5-1=\ds\prod_{i=0}^4(X-\zeta^i_5\e_+)(X-\zeta^i_5\e_{-}),\ \e_{\pm}=\frac{-1\pm\sqrt{5}}{2} .$$ Clearly, $|T|=12$. Further, 
it is easy to see that $r_{\infty_+}(T)\subset J[2](F)\subset J(F)_{{\rm tor}}$. 
The composition $C(\Q)\stackrel{r_{\infty_+}}{\lra}J(\Q)=J(\Q)_{{\rm tor}}\subset J(F)_{{\rm tor}}$ is 
injective, since $J(\Q)=J(\Q)_{{\rm tor}}$ as mentioned before. It yields a natural commutative diagram:
$$\xymatrix{ 
 & T\ar[r]^{r_{\infty_+}}\ar[d]_{\cap}   & J(F)_{{\rm tor}}\ar[d]^{\cap} \\
C(\Q)\ar@/^45pt/[urr]^{r_{\infty_+}}\ar[r]_{\subset} \ar[d]  & C(F) \ar[r]^{r_{\infty_+}}\ar[d]  & J(F) \ar[d]    \\
\widetilde{C}(\F_{11})\ar[r]_{=}& \widetilde{C}(\F_w)\ar[r]^{r_{\infty_+}} & \widetilde{J}(\F_w)
}$$  
where the vertical arrows to the bottom line stands for the reduction maps to 
good reductions $\widetilde{C}$ and $\widetilde{J}$ respectively. 
Notice that $J(F)_{{\rm tor}}$ can be extended to a finite etale group scheme in the 
Neron model of $J$ over $\O_F$. 
Therefore, the composition $J(F)_{{\rm tor}}\lra \widetilde{J}(\F_w)=\widetilde{J}(\F_{11}) $ of the most right vertical arrows 
is injective since the $F_w=\Q_{11}$ is trivially unramified over $\Q_{11}$ 
(this can be also checked by observing $11\nmid |J(\F_v)|$ for 
another finite place $v$ not lying over 11 of $F$ such that $J$ has good reduction at $v$). 
It follows from this that the reduction maps $C(\Q)\lra \widetilde{C}(\F_{11})$ and 
$T\lra \widetilde{C}(\F_w)=\widetilde{C}(\F_{11})$ are injective. 
By computation, $|\widetilde{C}(\F_{11})|=12$. Hence $T\stackrel{\sim}{\lra}\widetilde{C}(\F_w)=\widetilde{C}(\F_{11})$. 
Thus, the natural inclusion $C(\Q)\subset C(F)$ factors through $T$. Hence, $C(\Q)=C(\Q)\cap T=\{\infty_{\pm}\}$ 
as desired. 
\end{proof}

\subsection{On $A_5$ and $D_{10}$}
Next we consider the case of  $A_5$ or $D_{10}$ though the latter case is done already. This case can be excluded by observing the discriminant as below. 
Let $\psi\in \Q\setminus\{1\}$. Recall $f_\psi(x)=4x^5-5\psi x^4+1$ with the discriminant 
$D_\psi:=2^85^5(1-\psi^5)$. If $f_\psi$ is irreducible over $\Q$ and 
$D_\psi\not\in (\Q^\times)^2$, then we can exclude the cases of $D_{10}$ (see the proof of Theorem 1 of \cite{RYZ}) 
and $A_5$ (see Theorem 5 of \cite{HT}) for the Galois group of $f_\psi$ over $\Q$. 
To confirm this it suffices to prove the following. 
Let $C$ be a unique smooth completion of the affine hyperelliptic curve $W:Y^2=5(1-X^5)$ and 
$J$ be its Jacobian. 
Clearly, $C(\Q)=W(\Q)\cup\{\infty\}$.  
\begin{thm}\label{no-rat-D10} 
The affine hyperelliptic curve $W:Y^2=5(1-X^5)$ has no $\Q$-rational solution 
except for $(X,Y)=(1,0)$.  
\end{thm} 
\begin{proof}
By using Magma \cite{BCP} again, we see that $J(\Q)=J(\Q)_{{\rm tor}}$. Let $r_{\infty}:C(\bQ)\lra J(\bQ),\ 
Q\mapsto Q-\infty$. 
Working on $F=\Q(\zeta_5)$, the claim is similarly proved as in the proof of Theorem \ref{no-rat-F20}. 
In this case, we consider the set $T=\{(x,0)\in W(\bQ)\}\cup\{(0,\pm\sqrt{5})\}\cup\{\infty\}\subset C(F)$. 
Clearly, $r_\infty((x,0))\in J[2](F)$ for each $(x,0)\in T$ and $r_\infty((0,\pm\sqrt{5}))\in J[5](F)$ since $(y+\sqrt{5})(y-\sqrt{5})=5x^5$. 
It can be checked that the reduction map induces $T\stackrel{\sim}{\lra}\widetilde{C}(\F_{11})$ since 
$r_{\infty}(T)\subset J[10](F)\subset J(F)_{{\rm tor}}$ and $|T|=|\widetilde{C}(\F_{11})|=8$. 
The details are omitted. 
\end{proof}

\begin{rmk}Regarding above proofs, we apply Magma \cite{BCP} to check the Mordell-Weil rank of $J(\Q)$ is zero. 
Therefore, Theorem \ref{completed-version} is as reliable as its system thought the authors in this paper 
are responsible for the contents in this section in any points.  
\end{rmk}

\section{A variant}\label{variant}
In this section, we discuss a variant of the Dwork quintic family and results similar to  
main theorems (Theorem \ref{image} and Theorem \ref{mt-automorphy}). We omit the proof of them because they are proved without any changes. 

Let us consider the family of the Calabi-Yau threefolds over $K$ defined by 
\begin{equation}\label{x-v}
X_{\psi,\phi}: X^5_0+X^5_1+X^5_2+X^5_3+\phi X^5_4-5\psi X_0X_1X_2X_3X_4=0
\end{equation}
as a family of smooth projective hypersurface in $\mathbb{P}^4$. It is a twisted version of the Dwork quintic family. 
For each $\psi,\phi\in \overline{K}$ with $\phi(\psi^5 \phi-1)\neq 0$, the fiber is smooth. 
Let $Y_{\psi,\phi}$ be the singular mirror symmetry of $X_{\psi,\phi}$ which is defined by 
the closure of the smooth affine variety 
\begin{equation}\label{y-v}
  U_{\psi,\phi}:    x_1 + x_2 + x_3 + x_4 + \frac{\phi}{x_1x_2x_3x_4} - 5\psi = 0. 
\end{equation}
Then  a smooth mirror symmetry $W_{\psi,\phi}$ of $X_{\psi,\phi}$ is 
similarly constructed from $Y_{\psi,\phi}$ such that 
$W_{\psi,\phi}$ has good reduction at each finite place $v$ of $K$ 
satisfying $v\nmid 5\phi$ and $\psi^5 \phi-1$ is 
a $v$-adic unit. 
Let us introduce the quintic trinomial 
$$f_{\psi,\phi}(x):=4x^5-5\psi x^4+\phi$$
over $K$ and denote by $K_{f_{\psi,\phi}}$ the 
decomposition field of $f_{\psi,\phi}$ over $K$. 
By replacing $(x,\psi,\phi)$ with $(\frac{1}{x},a,b)$ where  
$(a,b)=(-\frac{4\psi}{5\phi},\frac{4}{\phi})$, 
we have $K_{f_{\psi,\phi}}=K_{g_{a,b}}$ for 
$$g_{a,b}(x)=x^5+ax+b.$$
Let $\br_{\psi,\phi,2}:G_K\lra \GSp_4(\F_2)$ be the mod 2 Galois representation 
associated to $H^3_{\text{\'et}}(Y_{\psi,\phi,\overline{K}},\F_2)$. Note that 
$H^3(Y_{\psi,\phi}(\C),\Z)$ (and also $H^3_{\text{\'et}}(Y_{\psi,\phi,\overline{K}},\Z_2)$) is torsion free since $Y_{\psi,\phi}(\C)$ is 
isomorphic to a smooth fiber of the Dwork quintic family.  
\begin{thm}\label{image-v}
Let $\psi,\phi\in K$ with $\phi(\psi^5 \phi-1)\neq 0$. 
Assume that $f_{\psi,\phi}$ is irreducible over $K$. Then it holds that 
\begin{enumerate}
\item 
${\rm Im}(\br_{\psi,\phi,2})\simeq {\rm Gal}(K_{f_{\psi,\phi}}/K)$ and ${\rm Im}(\br_{\psi,\phi,2})$ contains no element of type $(3,3)$  
under the fixed isomorphism ${\rm GSp}_4(\F_2)\simeq S_6$. 
In particular, the image is regarded as a subgroup of the 5-th symmetric group $S_5$ whose order is divisible by five; 
\item $\br_{\psi,\phi,2}$ is irreducible over $\F_2$ and hence $\br_{\psi,\phi,2}\simeq \br^{{\rm ss}}_{\psi,\phi,2}$. Further, it is absolutely irreducible in which case the image is isomorphic to 
$F_{20}=C_4\ltimes C_5,A_5$ or $S_5$ unless 
the image is isomorphic to $C_5$ or $D_{10}=C_2\ltimes C_5$. 
\end{enumerate}
\end{thm}

\begin{thm}\label{mt-automorphy-v}
Keep the assumption on $f_{\psi,\phi}$ as in Theorem \ref{image-v}.  
Suppose that $K=F$ is a totally real field and $f_{\psi,\phi}$ is irreducible over $F$. 
Assume that ${\rm Gal}(F_{f_{\psi,\phi}}/F)\simeq F_{20},\ A_5$ or $S_5$ and each complex conjugation in ${\rm Gal}(F_{f_{\psi,\phi}}/F)$ 
corresponds to an element of type $(2,2)$. Further assume $[F:\Q]$ is even if ${\rm Gal}(F_{f_{\psi,\phi}}/F)\simeq A_5$. Let $M/F$ be the totally real quadratic extension 
associated to the kernel of $G_F\lra {\rm Im}(\br_{\psi,\phi,2})\lra S_5\stackrel{\sgn}{\lra}\{\pm 1\}$ if ${\rm Gal}(F_{f_{\psi,\phi}}/F)\simeq S_5$ and 
$M=F$ otherwise. Put $d=[M:\Q]$.  
Then there exists a holomorphic Hilbert-Siegel Hecke eigen cusp form $h$ on $\mathcal{H}^d_2$ of parallel weight three such that  $\br_{\psi,\phi,2}|_{G_M}\simeq \br_{h,2}$.  
\end{thm}

\subsection{An example}
Let $(\psi,\phi)=(-\frac{4}{5},-\frac{1}{4})$ which satisfies the condition 
in both of Theorem \ref{image-v} and Theorem \ref{mt-automorphy-v} 
with the Galois group ${\rm Gal}(\Q_{f_{\psi,\phi}}/\Q)\simeq S_5$. 
Notice that $\Q_{f_{\psi,\phi}}=\Q_g$ 
with $g(x)=x^5-x-1$ and the reciprocity of this polynomial has been studied in 
\cite{KRW}. Let $M/\Q$ be the quadratic extension in 
Theorem \ref{mt-automorphy-v}.  
Then $M=\Q(\sqrt{19\cdot 151})$. It follows from 
Theorem \ref{mt-automorphy-v}  that  
a holomorphic Hilbert-Siegel Hecke eigen cusp form $h$ on $\mathcal{H}^2$ of 
weight $(3,3)$ 
which characterizing mod 2 reciprocity of $f_{\psi,\phi}(x)$ (and also of $g(x)$) over $M$.


\begin{thebibliography}{99}

\bibitem{AGMC}A. Ash, P.-E. Gunnells, and M. McConnell, Mod 2 homology for ${\rm GL(4)}$ and Galois representations. J. Number Theory 146 
(2015), 4-22.

\bibitem{BL}T. Barnet-Lamb, Meromorphic continuation for the zeta
function of a Dwork hypersurface, Algebra Number Theory 4 (2010), no. 7, 839-854.

\bibitem{BGHT} T. Barnet-Lamb, D. Geraghty, M. Harris, and R. Taylor, A family of Calabi-Yau varieties and potential automorphy II. Publ. Res. Inst. Math. Sci. 47 (2011), no. 1, 29-98.

\bibitem{Ba} V.-V. Batyrev, 
Dual polyhedral and mirror symmetry for Calabi-Yau hypersurfaces in toric varieties. 
J. Algebraic Geom. 3 (1994), no. 3, 493-535.

\bibitem{BPPTVY}A. Brumer, A. Pacetti, C. Poor, G. Tornaria, J. Voight, and D. Yuen, 
On the paramodularity of typical abelian surfaces. 
Algebra Number Theory 13 (2019), no. 5, 1145-1195.

\bibitem{Ba99} \bysame, 
Birational Calabi-Yau $n$-folds have equal Betti numbers. New trends in algebraic geometry (Warwick, 1996), 1-11, 
London Math. Soc. Lecture Note Ser., 264, Cambridge Univ. Press, Cambridge, 1999. 


\bibitem{BK} V.-V. Batyrev and M. Kreuzer, 
Integral cohomology and mirror symmetry for Calabi-Yau $3$-folds. Mirror symmetry. V, 255-270, 
AMS/IP Stud. Adv. Math., 38, Amer. Math. Soc., Providence, RI, 2006. 


\bibitem{Berger}T. Berger, Oddness of residually reducible Galois representations. 
Int. J. Number Theory 14 (2018), no. 5, 1329-1345.

\bibitem{BCP}W. Bosma, J. Cannon, and C. Playoust, The Magma algebra system. I. The user
language. J. Symbolic Comput., 24(3-4):235-265, 1997. Computational algebra and number
theory (London, 1993).

\bibitem{BDJ}K. Buzzard, F. Diamond, and F. Jarvis, 
On Serre's conjecture for mod $\ell$ Galois representations over totally real fields. 
Duke Math. J. 155 (2010), no. 1, 105-161.

\bibitem{BG}K. Buzzard and T. Gee, The conjectural connections between automorphic representations and Galois representations. Automorphic forms and Galois representations. Vol. 1, 135-187, 
London Math. Soc. Lecture Note Ser., 414. 


\bibitem{Ca}F. Calegari,
The Artin conjecture for some $S_5$-extensions. 
Math. Ann. 356 (2013), no. 1, 191-207.



\bibitem{COR} P. Candelas, X. de la Ossa, and F. Rodriguez-Villegas, 
Calabi-Yau manifolds over finite fields. II. Calabi-Yau varieties and mirror symmetry (Toronto, ON, 2001), 
121-157, Fields Inst. Commun., 38, Amer. Math. Soc., Providence, RI, 2003.

\bibitem{DS}P. Deligne and J.-P. Serre, Formes modulaires de poids 1. 
Ann. Sci. Ecole Norm. Sup. (4) 7 (1974), 507-530 (1975). 

\bibitem{Dem}L. Dembele, Explicit computations of Hilbert modular forms on $\Q(\sqrt{5})$. 
Experiment. Math. 14 (2005), no. 4, 457-466.

\bibitem{DLP}L. Dembele, D. Loeffler, and A. Pacetti, 
Non-paritious Hilbert modular forms. 
Math. Z. 292 (2019), no. 1-2, 361-385.

\bibitem{Di}M. Dimitrov, Galois representations modulo $p$ and cohomology of Hilbert modular varieties. Ann. Sci. Ecole Norm. Sup. (4) 38 (2005), no. 4, 505-551. 

\bibitem{Dimi}\bysame, Arithmetic aspects of Hilbert modular forms and varieties. Elliptic curves, Hilbert modular forms and Galois deformations, 119-134, Adv. Courses Math. CRM Barcelona, Birkha\"user/Springer, Basel, 2013.


\bibitem{Du}D. S. Dummit,
Solving solvable quintics.
Math. Comp. 57 (1991), no. 195, 387-401.

\bibitem{E}B. Edixhoven, Computing the residual Galois representations. Computational aspects of modular forms and Galois representations, 371-382, Ann. of Math. Stud., 176, Princeton Univ. Press, Princeton, NJ, 2011. 

\bibitem{GT}W.-T. Gan and S. Takeda, The local Langlands conjecture for $\mathrm{GSp}(4)$. Ann. of Math. (2) 173 (2011), no. 3, 
1841-1882. 

\bibitem{GAP}GAP - Groups, Algorithms, Programming -
a System for Computational Discrete Algebra.  

\bibitem{GHS} T. Gee, F. Herzig and D. Savitt,  General Serre weight conjectures, 
J. Eur. Math. Soc. (JEMS) 20 (2018), no. 12, 2859-2949.


\bibitem{GeeT} T. Gee and O. Ta\"ibi, Arthur's multiplicity formula for 
$\mathrm{GSp}_4$ and restriction to $\mathrm{Sp}_4$. J. Ec. polytech. Math. 6 (2019), 469-535.


\bibitem{Go}P. Goutet, On the zeta function of a family of quintics. J. Number Theory 130 (2010), no. 3, 478-492. 

\bibitem{HT}K. Hashimoto and H. Tsunogai, Generic polynomials over Q with two parameters for the transitive groups of degree five. Proc. Japan Acad. Ser. A Math. Sci. 79 (2003), no. 9, 142-145.

\bibitem{Jarvis}F. Jarvis, On Galois representations associated to Hilbert modular forms. J. Reine Angew. Math. 491 (1997), 199-216. 

\bibitem{KRW} C-B. Khare, A. F. La Rosa and G. Wiese, 
Splitting fields of $X^n-X-1$ (particularly for $n=5$), prime decomposition and modular forms, arXiv:2206.08116.

\bibitem{KKMS} 
G.-R. Kempf, F.-F. Knudsen, D. Mumford, and B. Saint-Donat, 
Toroidal embeddings. I. Lecture Notes in Mathematics, Vol. 339. Springer-Verlag, Berlin-New York, 1973.

\bibitem{KT}C.-B. Khare and J-A. Thorne, Automorphy of some residually $S_5$ Galois representations. Math. Z. 286 (2017), no. 1-2, 399-429. 



\bibitem{KS}H.-H. Kim and F. Shahidi, Symmetric cube $L$-functions for $\mathrm{GL}_2$ are entire. Ann. of Math. (2) 150 (1999), no. 2, 645-662. 


\bibitem{Kisin-finite} M. Kisin, Moduli of finite flat group schemes, and modularity. 
Ann. of Math. (2) 170 (2009), no. 3, 1085-1180. 

\bibitem{Kisin}\bysame, 
Modularity of 2-adic Barsotti-Tate representations. 
Invent. Math. 178 (2009), no. 3, 587-634.

\bibitem{Mok}C.-P. Mok, Galois representations attached to automorphic forms on $GL_2$ over CM fields. Compos. Math. 150 (2014), no. 4, 523-567.

\bibitem{P}S. Patrikis, 
On the sign of regular algebraic polarizable automorphic representations. 
Math. Ann. 362 (2015), no. 1-2, 147-171.

\bibitem{PT}S. Patrikis and R. Taylor, Automorphy and irreducibility of some
$l$-adic representations, Compos. Math. 151 (2015) 207-229. 


\bibitem{Ra}D. Ramakrishnan, Modularity of the Rankin-Selberg $L$-series, and multiplicity one for ${\rm SL(2)}$. Ann. of Math. (2) 152 (2000), no. 1, 45-111.

\bibitem{R}B. Roberts, Global L-packets for $\mathrm{GSp}(2)$ and theta lifts. Doc. Math. 6 (2001), 247-314. 

\bibitem{RS}B. Roberts and R. Schmidt, {\em Local Newforms for ${\rm GSp}(4)$}, Lecture Notes in Math. vol 1918, 
Springer-Verlag, Berlin-New York, 2007.

\bibitem{RT}J.-D. Rogawski and J-B. Tunnell, On Artin $L$-functions associated to Hilbert modular forms of weight one. Invent. Math. 74 (1983), no. 1, 1-42.

\bibitem{RYZ}G. Roland, N. Yui, and D. Zagier, 
A parametric family of quintic polynomials with Galois group $D_5$. J. Number Theory 15 (1982), no. 1, 137-142. 

\bibitem{Sasaki}S. Sasaki, Integral models of Hilbert modular varieties in the ramified case, deformations of modular Galois representations, and weight one forms. Invent. Math. 215 (2019), no. 1, 171-264.

\bibitem{Se1}J.-P. Serre, 
Modular forms of weight one and Galois representations. Algebraic number fields: $L$-functions and Galois properties 
(Proc. Sympos., Univ. Durham, Durham, 1975), pp. 193-268. Academic Press, London, 1977.



\bibitem{Taylor1}R. Taylor, On the meromorphic continuation of degree two $L$-functions.
Doc. Math. 2006, Extra Vol., 729-779.  

\bibitem{T1}J.-A. Thorne, 
A 2-adic automorphy lifting theorem for unitary groups over CM fields. 
Math. Z. 285 (2017), no. 1-2, 1-38.


\bibitem{Wa}D. Wan,  
Mirror symmetry for zeta functions. 
With an appendix by C. Douglas Haessig. AMS/IP Stud. Adv. Math., 38, Mirror symmetry. V, 159-184, Amer. Math. Soc., Providence, RI, 2006. 




\bibitem{Wei1}R. Weissauer, Four dimensional Galois representations, Ast\'erisque. {\bf 302} (2005), 67-150.

\bibitem{Wei2} \bysame, {\em Endoscopy for $\mathrm{GSp}(4)$ and the cohomology of Siegel modular threefolds.} 
Lecture Notes in Mathematics 1968. Springer-Verlag, Berlin, 2009.

\bibitem{Wei3} \bysame, Existence of Whittaker models related to four dimensional symplectic Galois representations,
Modular forms on Schiermonnikoog, 285-310, Cambridge Univ. Press, Cambridge, 2008.

\bibitem{TY}T. Yamauchi, Automorphy of certain 2-adic Galois representations of rank 4 with 
the residual image $A_5$, in preparation 2020.  

\end{thebibliography}
\end{document}